\documentclass{amsart}
\usepackage{lmodern}
\usepackage{url}

\usepackage{amssymb}
\usepackage{enumitem}
\usepackage[foot]{amsaddr}
\usepackage[numbers]{natbib}
\usepackage[hidelinks]{hyperref}

\theoremstyle{plain}
\newtheorem{theorem}{Theorem}[section]
\newtheorem{proposition}[theorem]{Proposition}
\newtheorem*{proposition*}{Proposition}
\newtheorem{lemma}[theorem]{Lemma}
\newtheorem*{lemma*}{Lemma}
\newtheorem{corollary}[theorem]{Corollary}

\newtheorem{conjecture}[theorem]{Conjecture}
\theoremstyle{definition}

\newtheorem*{definition*}{Definition}

\newtheorem{problem}[theorem]{Problem}
\theoremstyle{remark}

\DeclareMathOperator{\SO}{SO}
\DeclareMathOperator{\SU}{SU}
\DeclareMathOperator{\SL}{SL}

\DeclareMathOperator{\Ad}{Ad}
\DeclareMathOperator{\Id}{Id}

\title{Minimal doubling for small subsets in compact Lie groups}
\author{Simon Machado}
\address{ETH Zurich, R\"amistrasse 101, 8092 Zurich, Switzerland}
\email{smachado@ethz.ch}
\date{}

\begin{document}

\begin{abstract}
    We prove an asymptotically sharp bound for the minimal doubling of a small measurable subset of a compact connected Lie group, settling a conjecture of Breuillard and Green. Namely, let $G$ be a compact connected Lie group of positive dimension $d_G$ with Haar probability measure $\mu_G$, and let $\epsilon > 0$. We show that for all compact subsets $A$ of sufficiently small measure we have 
    \[\mu_G(A^2) \geq \left(2^{d_G-d_M} - \epsilon\right)\mu_G(A),\]
    where $d_M$ is the maximal dimension of a proper closed subgroup. This recovers, by a different method, the small-measure doubling theorem of Jing--Tran--Zhang for $G=\SO_3(\mathbb{R})$. More generally, for every $\epsilon>0$ and any pair of non-empty compact subsets $A,B$ of sufficiently small measure, we prove the Brunn--Minkowski-type inequality:
    \[ \mu_G(AB)^{\frac{1}{d_G-d_M}} \geq (1-\epsilon)\left( \mu_G(A)^{\frac{1}{d_G-d_M}} + \mu_G(B)^{\frac{1}{d_G-d_M}}\right).\]

    We also prove an inverse theorem in the semi-simple case: for any fixed $K\geq1$ and $\delta>0$, every compact $K$-approximate subgroup of sufficiently small positive Haar measure is covered by $C(d_G,K)$ translates of the $\delta$-neighbourhood of a proper closed connected subgroup.
    Our approach combines this inverse theorem with optimal transport and its recent applications to the Brunn--Minkowski inequality.

\end{abstract}
\maketitle

\section{Introduction}
Let $G$ be a compact connected Lie group of positive dimension $d_G$ and let $\mu_G$ be the Haar probability measure, i.e. the unique $G$-invariant Borel probability measure on $G$. The \emph{doubling constant} of a measurable subset $A \subset G$ of positive measure is defined as the ratio, with inner Haar measure used if the product is not measurable,
\[\frac{\mu_G(A^2)}{\mu_G(A)},\]
where $A^2:=\{a_1a_2 \vert a_1, a_2 \in A\}$ is the Minkowski product. The doubling constant is an important object of study in both additive combinatorics \cite{IntroAppGroups14Breuillard} and geometric analysis \cite{BrunnMinkowski02Gardner}. The minimal value it can take in a given group $G$ has been a long-established subject of research. A foundational general lower bound for connected locally compact groups is due to Kemperman \cite{Kemperman64}. For sets of finite positive Haar measure in Euclidean spaces and simply connected nilpotent Lie groups of dimension $d_G$, the doubling constant is at least $2^{d_G}$ \cite{zbMATH06005474}. Recent work of Jing, Tran and Zhang \cite{zbMATH07715451} yields similar inequalities in many non-compact Lie groups.

This paper is concerned with doubling inequalities in compact Lie groups. Taking $A=G$ gives doubling constant $1$, so a lower bound greater than $1$ requires a restriction on the size of $A$. Small measure, however, does not imply that $A$ is contained in a small ball: a thin neighbourhood of a proper subgroup has small measure while extending along the whole subgroup. The local dimension of $G$ therefore does not by itself determine the least possible doubling. Breuillard and Green conjectured that doubling for small sets in $\SO_3(\mathbb{R})$ must be at least $4 - \epsilon$. In higher dimensions, their conjecture becomes:

\begin{conjecture}[Breuillard--Green conjecture for compact Lie groups, \cite{Green100,jing2023measure}]\label{Conjecture: Breuillard--Green}
    Let $G$ be a compact connected Lie group of positive dimension $d_G$ and let $d_M$ be the maximal dimension of a proper closed subgroup. Then
    \[\lim_{\alpha\downarrow0}\inf\left\{\frac{\mu_G(A^2)}{\mu_G(A)}\ \middle\vert\ A\subset G\text{ compact},\ 0<\mu_G(A)\leq\alpha\right\}=2^{d_G-d_M}.\]
\end{conjecture}

The exponent $d_G-d_M$ counts the directions transverse to a subgroup of maximal dimension. A thin neighbourhood of such a subgroup already fills the subgroup directions, so doubling acts to first order in the remaining directions. This explains why these neighbourhoods are the natural model for small doubling.

For example, in $G=\SO_3(\mathbb{R})$ let $H\simeq\SO_2(\mathbb{R})$ be the subgroup of rotations fixing a given axis. The neighbourhood $H_\delta$ has two transverse directions, so its measure is asymptotic to a positive constant times $\delta^2$. Since $(H_\delta)^2=H_{2\delta}$, its doubling tends to $4$ as $\delta$ tends to zero, whereas the doubling of a small ball in $G$ tends to $8$. Table~\ref{Table: Maximal subgroup codimensions} in Appendix~\ref{Section: Subgroup of maximal dimension} lists the corresponding codimensions for compact almost simple Lie groups.

When $G$ admits a non-trivial torus quotient, the conjecture follows from Kemperman's inequality \cite{Kemperman64}. Little progress was made on the Breuillard--Green conjecture beyond that case until a recent breakthrough of Jing, Tran and Zhang \cite{jing2023measure} building upon previous work and model-theoretic tools. They solved the original Breuillard--Green conjecture (corresponding to $G =\SO_3(\mathbb{R})$ above), but the sharp bound for general $G$ remained open.

The main result of this paper is a proof of the general case of this conjecture:

\begin{theorem}[Minimal doubling in compact Lie groups]\label{Theorem: Expansion inequality}
 Let $G$ be a compact connected Lie group of positive dimension and let $\epsilon > 0$. There is $\alpha > 0$ such that for all compact $A \subset G$ with $\mu_G(A) \leq \alpha$ we have \[\mu_G(A^2) \geq \left(2^{d_G-d_M} - \epsilon\right)\mu_G(A),\]
where $d_G$ denotes the dimension of $G$ and $d_M$ denotes the maximal dimension of a proper closed subgroup. 
\end{theorem}

The proof also gives a bound for the error in terms of the measure of $A$. Suppose that $G$ is semi-simple and that its Lie algebra has $n$ simple factors, counted with multiplicity. There are constants $c,C>0$, depending only on $d_G$, such that every compact subset $A$ of sufficiently small positive measure satisfies
\[\mu_G(A^2)\geq\left(2^{d_G-d_M}-\frac{C}{\bigl(\log^{\circ(n+1)}(1/\mu_G(A))\bigr)^c}\right)\mu_G(A),\]
where $\log^{\circ j}$ denotes the $j$-fold iterate of the natural logarithm, and the measure is small enough that all these logarithms are defined and positive. In particular, for an almost simple group the error is $O_{d_G}\bigl((\log\log(1/\mu_G(A)))^{-c}\bigr)$.

The repeated logarithms come from the measure threshold in the inverse theorem, while the estimate near a subgroup has a power error in the neighbourhood radius. We record the deduction at the end of \S \ref{Section: Expansion of small subsets}.

When $G$ is not semi-simple, Kemperman's inequality gives $\mu_G(A^2)\geq2\mu_G(A)$ for $\mu_G(A)\leq1/2$, so no error is needed.

Some assumption of small measure is necessary, since $\mu_G(G^2)=\mu_G(G)$. One expects more precisely a correcting term of order $\mu_G(A)^{\frac{2}{d_G-d_M}}$, and in the semi-simple case that order would be optimal: if $H$ has maximal dimension among proper closed subgroups, the subsets $H_{\delta}$ satisfy $\mu_G(H_{\delta}) \asymp \delta^{d_G-d_M}$ and have doubling $2^{d_G-d_M}-\Theta(\delta^2)$ (by the small-ball volume formula on the normal homogeneous space $G/H$, whose scalar curvature is positive in this case). An error bound of this order remains open in the present argument. 

As is often the case, the doubling inequality (Theorem \ref{Theorem: Expansion inequality}) stems from a much more general lower bound on the measure of product-sets. The lower bound we obtain takes the form of a Brunn--Minkowski-type inequality:

\begin{theorem}[Brunn--Minkowski in compact Lie groups]\label{Theorem: BM in compact groups}
    Let $G$ be a compact connected Lie group of positive dimension. Let $\epsilon >0$. There is $\alpha > 0$ such that if $A,B$ are non-empty compact subsets of $G$ with $\mu_G(A), \mu_G(B) \leq \alpha$,  then \[\mu_G(AB)^{\frac{1}{d_G-d_M}} \geq (1-\epsilon)\left(\mu_G(A)^{\frac{1}{d_G-d_M}} + \mu_G(B)^{\frac{1}{d_G-d_M}}\right),\]
where $d_G$ denotes the dimension of $G$ and $d_M$ denotes the maximal dimension of a proper closed subgroup.
\end{theorem}

By inner regularity, the two theorems also hold for Haar-measurable sets, with the same thresholds and with inner Haar measure for products that are not known to be measurable. Indeed, approximate each set from within by compact subsets, using a singleton when a non-empty input has measure zero.

The Brunn--Minkowski-type inequality is asymptotically sharp, and in the semi-simple case the presence of the correcting term $\epsilon$ is also necessary. The quantitative dependence on the measures, especially when $\mu_G(B)$ becomes small relative to $\mu_G(A)$, raises further questions. We do not pursue these here. The asymmetric problem is developed further in \cite{machado2025asymmetric}.

We also establish an inverse theorem for approximate subgroups of small measure, which is the main technical contribution of this paper. Recall that a $K$-approximate subgroup is a symmetric subset $\Lambda$ containing $e$ such that $\Lambda^2$ is covered by at most $K$ left translates of $\Lambda$. We write $H_\delta$ for the $\delta$-neighbourhood of a subgroup $H$.

\begin{theorem}[An inverse theorem for approximate subgroups of small measure]\label{Proposition: approximate subgroups of small measure}
    Let $G$ be a compact connected semi-simple Lie group of positive dimension, equipped with the metric fixed in \S \ref{Subsection: Notation}, and let $K \geq 1$. For all $\delta > 0$ there is $\epsilon > 0$ such that if $\Lambda \subset G$ is a compact $K$-approximate subgroup with $e \in \Lambda$, $\mu_G(\Lambda) > 0$ and $\mu_G(\Lambda) \leq \epsilon$ then there is a proper closed connected subgroup $H$ such that $\Lambda$ is covered by $C(d_G,K)$ translates of $H_{\delta}$.
\end{theorem}

Besides its use in the proof of Theorem \ref{Theorem: BM in compact groups}, this result is of independent interest: it applies for every fixed $K$, with a covering bound independent of $\delta$. The measure threshold may depend on $\delta$.

Theorem \ref{Proposition: approximate subgroups of small measure} is a continuous analogue of Jordan's theorem: in a fixed compact connected Lie group, every finite subgroup is covered by a bounded number of cosets of a torus, with the bound depending only on the ambient group. In the semi-simple case the torus is proper. Breuillard and Green established a version for finite approximate subgroups of unitary groups \cite{BGUnitary12}. Here we consider compact approximate subgroups of small positive Haar measure. Our method relies on the shrinking commutator argument recalled in \S~\ref{Subsection: Jordan argument}, which we can exploit by analysing discretisations of approximate subgroups at various scales.

Even when specialised to $\SO_3(\mathbb{R})$, our proof provides a new approach to the result of \cite{jing2023measure} based on geometric ideas related to Jordan's theorem, the structure of approximate groups and the optimal transport approach to the Brunn--Minkowski inequality. This new approach provides further insight into the structure of sets with doubling close to the optimum. Recall that the relevance of the constant $2^{d_G - d_M}$ was explained in \cite{Green100,jing2023measure} by the fact that neighbourhoods of proper subgroups of maximal dimension have doubling asymptotic to $2^{d_G-d_M}$. This suggests a converse: sets with doubling close to the optimum should concentrate near a coset of a proper subgroup of maximal dimension. We address this stability question in companion work, the methods from the present paper being the main engine there as well.

\subsection{Historical background}
The first results on the doubling constant can be traced back to fundamental papers of Brunn \cite{Brunn87}, Minkowski \cite{Minkowski68} and Lusternik \cite{Lusternik35} regarding Minkowski sums in $\mathbb{R}^{d_G}$. Their works unveiled a natural relationship between minimal doubling and dimension: the base-$2$ logarithm of the minimal doubling constant equals the dimension $d_G$. This tells us that minimal doubling is intimately linked to the local behaviour of the set considered. These inequalities and their generalisations have been a fundamental tool of geometric analysis in relation to isoperimetric inequalities. We refer to the survey \cite{BrunnMinkowski02Gardner} for this and more.  

As a counterpart, the seminal work of Freiman \cite{Freiman73} showed that doubling is also linked to global phenomena. Freiman's work showed quantitatively that finite subsets of $\mathbb{Z}$ exhibiting small doubling are contained in generalised arithmetic progressions of bounded rank and controlled size, again relating small doubling to a dimensional constant. Many works have improved upon Freiman's results. They were extended to other abelian and non-abelian groups (e.g. \cite{GreenRuzsa07, Helfgott08}) culminating in a general classification result for finite approximate subgroups due to Breuillard, Green and Tao \cite{BGT12}. The quantitative aspects of Freiman's results were also improved towards the polynomial Freiman--Ruzsa conjecture \cite{Sanders13Survey, gowers2023conjecture}.

The interaction between these two fields recently gained traction following the seminal work of Figalli and Jerison on the quantitative stability problem of the Brunn--Minkowski inequality \cite{FigalliJerison15}. They observed there that the one-dimensional stability could be seen as a consequence of Freiman's $3k-4$ theorem \cite{Freiman59} and used it as the starting point of their proof. Following this, further quantitative stability results were obtained \cite{FigalliICM14, böröczky2022quantitative} culminating in the recent sharp bounds obtained by Figalli, van Hintum and Tiba \cite{figalli2023sharp}. Van Hintum, Keevash and Tiba \cite{vanhintum2023ruzsas} proved a discrete Brunn--Minkowski inequality in $\mathbb{Z}^n$, resolving a conjecture of Ruzsa. We show in this article that techniques from the structure theory of approximate groups and the optimal transport proof of the Brunn--Minkowski inequality can be combined to prove the Breuillard--Green conjecture for compact Lie groups (Theorem \ref{Theorem: Expansion inequality}). Ellis, Kindler, Lifshitz and Minzer \cite{ellis2024product} study product-set estimates for large subsets of compact Lie groups, complementing the small-measure regime considered here.

\subsection{A look at the methods}

The proof of Theorem \ref{Theorem: BM in compact groups}, and hence of Theorem \ref{Theorem: Expansion inequality}, has to account for sets that are small in measure but extend along a proper subgroup. Kemperman's inequality handles groups with a non-trivial torus quotient, so we may assume that $G$ is semi-simple.

The local result is an extension of the Brunn--Minkowski inequality to subsets in small enough neighbourhoods of the identity: 

\begin{theorem}[Local Brunn--Minkowski inequality]\label{Theorem: Local Brunn--Minkowski}
Let $G$ be a compact connected Lie group of positive dimension $d_G$, equipped with a bi-invariant Riemannian distance $d$. There are $C(G,d)>0$ and $\rho_0>0$ such that for all $0<\rho<\rho_0$ and all non-empty compact subsets $A,B\subset B_G(e,\rho)$,
\[\mu_G(AB)^{\frac{1}{d_G}}\geq (1-C(G,d)\rho^2)\left(\mu_G(A)^{\frac{1}{d_G}}+\mu_G(B)^{\frac{1}{d_G}}\right).\]
In the semi-simple case, with the metric fixed in \S \ref{Subsection: Notation}, both $C$ and $\rho_0$ can be chosen to depend only on $d_G$.
\end{theorem}

To establish this result we adapt a strategy described in particular in \cite[pp.~242--244]{FigalliICM14} proving Brunn--Minkowski-type inequalities thanks to optimal transport.

To prove Theorem~\ref{Theorem: BM in compact groups}, and hence the main doubling theorem, we must turn the local inequality into a bound for arbitrary sets $A,B$ of small measure. For pairs that do not already satisfy the required bound, Tao's product-set inverse theorem (Proposition~\ref{Proposition: Tao}) covers $A$ and $B$ by left and right translates of a common \emph{approximate subgroup} $\Lambda$ of small measure. Our inverse theorem (Theorem~\ref{Proposition: approximate subgroups of small measure}) then gives covers by boundedly many translates of a thin neighbourhood $H_\delta$ of a proper closed connected subgroup $H$. There are two issues in using these covers: obtaining the sharp inequality near $H$, where the sets need not lie in a small ball, and combining the estimates without losing a factor depending on the number of translates. For the first, we average the local Brunn--Minkowski inequality (Theorem \ref{Theorem: Local Brunn--Minkowski}) along $H$ using double counting. The volume growth of balls in $H$ enters through this averaging and yields the codimension exponent $1/(d_G-\dim H)$ in Proposition~\ref{Proposition: BM close to a subgroup}. For the second, we refine the covers so that products of a fixed piece of one set with distinct pieces of the other are disjoint. A comparison of the proportions of $A$ and $B$ contained in their respective pieces then allows the estimates near $H$ to be summed with no loss in the constant. Since $\dim H\leq d_M$, this gives the required inequality.

The proofs also give quantitative bounds, although their dependence on the parameters is far from optimal. 

\subsection{Outline of the paper}
We start in \S \ref{Section: Local Brunn--Minkowski via optimal transport} by establishing the local Brunn--Minkowski inequality (Theorem \ref{Theorem: Local Brunn--Minkowski}) thanks to an optimal transport argument. We follow up in \S \ref{Section: Approximate subgroups of small measure} by showing that small approximate subgroups are covered by boundedly many translates of neighbourhoods of proper subgroups (Theorem \ref{Proposition: approximate subgroups of small measure}). We conclude in \S \ref{Section: Expansion of small subsets} the proof of Theorems \ref{Theorem: Expansion inequality} and \ref{Theorem: BM in compact groups} by relying on the previous two sections and a double counting argument. 

We end this paper in \S \ref{Section: Closing remarks} with questions about minimisers and the error close to a subgroup, followed by an application to the growth of small subsets.

\subsection{Notation}\label{Subsection: Notation}

Given two subsets $X, Y$ of a group $G$ we denote by $XY$ the set of products $\{xy \vert x \in X, y \in Y\}$ and we write $X^{-1}:=\{x^{-1} \vert x \in X \}$. For all $n \in \mathbb{N}$ define inductively $X^0:=\{e\}$, $X^1:=X$, $X^2:=XX$ and $X^{n+1}:=XX^n$.

Whenever a product is not known to be measurable, its measure is understood as inner Haar measure, the supremum of the measures of its compact subsets. 

Unless a result explicitly states otherwise, $G$ denotes a compact connected semi-simple Lie group of positive dimension $d_G$. We say a compact connected Lie group is semi-simple if it has no non-trivial \emph{connected} normal closed abelian subgroup, equivalently if the identity component of its centre is trivial. Here almost simple means that the Lie algebra is simple, so an almost simple group may have non-trivial finite centre. The word connected in the condition on the normal subgroup cannot be dropped: $\SU_2(\mathbb{C})$ has centre $\{\pm 1\}$. Theorems \ref{Theorem: Expansion inequality} and \ref{Theorem: BM in compact groups} are stated for compact connected $G$. Lemma \ref{Lemma: Reduction to the semi-simple case} below reduces them to the semi-simple case, so we lose nothing by this convention. 

Compact Lie groups can be equipped with a bi-invariant (i.e. invariant under right and left translations) Riemannian metric as well as a Haar probability measure $\mu_G$ which is a bi-invariant Borel probability measure, see \cite{Compact07Sepanski} for this and more. We fix the metric as follows. Write $G=G_1\cdots G_n$ as an almost direct product of its compact almost simple factors. On each $G_i$ choose the bi-invariant Riemannian metric whose distance has diameter $1/n$, take the orthogonal sum of these metrics on $\prod_iG_i$, and give $G$ the quotient metric under the finite central multiplication map. Thus the product distance is $(\sum_i d_i^2)^{1/2}$, and $G$ has diameter at most $1/\sqrt n\leq 1$. 

There are only finitely many compact semi-simple groups in bounded dimension up to isomorphism. Since the scale on each almost simple factor is now fixed, the geometric constants used below depend only on $d_G$. The quotient metrics on factor groups are comparable, with constants depending only on $d_G$, to the same normalisation on those groups. We use that comparison whenever we apply a result to a factor quotient.

Write $\langle\cdot,\cdot\rangle$ for the resulting $\Ad(G)$-invariant scalar product on the Lie algebra $\mathfrak{g}$. We write $\lVert \cdot \rVert $ for the associated norm and $d$ for the bi-invariant Riemannian distance it induces on $G$. All orthogonal complements in $\mathfrak{g}$ are taken with respect to $\langle \cdot, \cdot \rangle$. We will often use the fact that balls centred at $e$ for the distance $d$ are normal, and that $d$ is geodesic: the one-parameter subgroups of $G$ are the geodesics through $e$ and are minimising below the injectivity radius, so that $\lVert X\rVert =d(e,\exp(X))$ for $\lVert X\rVert $ smaller than the injectivity radius of $G$ at $e$, and $B_G(e,r+s) \subset B_G(e,r)B_G(e,s)$ for all $r,s>0$.  

A maximal proper closed connected subgroup means one that is maximal among proper closed connected subgroups: if $H \subset L \subset G$ and $L$ is closed and connected, then $L=H$ or $L=G$. This is distinct from having maximal dimension. Given a subgroup $H$ and $\delta > 0$ write $H_\delta$ for the $\delta$-neighbourhood of $H$. It is defined by 
\[ H_\delta := \{g \in G \vert \exists h \in H, d(h,g) < \delta\}=B_G(e,\delta)H = HB_G(e,\delta),\]
where $B_G(e, \delta)$ denotes the ball of radius $\delta$ centred at $e$. The identity component of $H$ is written $H^{\circ}$. A subscript on $B$ naming a space is followed by a centre and a radius, as in $B_G(e,r)$, $B_H(h,\rho)$ and $B_{G/H}(gH,\delta)$. For a normed vector space $W$, we use $B_W(0,r)$ for the norm ball. We also use the abbreviation $B_r:=B_G(e,r)$ and write $\overline{B}_r$ for the corresponding closed ball.

We write $d_M$ for the maximal dimension of a proper closed subgroup of $G$, and, whenever a specific proper closed subgroup $H$ has been fixed, $d_H := \dim H$. Thus $d_H \leq d_M$, with equality exactly when $H$ has maximal dimension among proper closed subgroups.

Throughout the paper, we will deal with several parameters and constants that depend on one another. When needed, we indicate the dependence in parentheses: a constant $C(d_G)> 0$ is a constant depending on the parameter $d_G$. We will usually not keep track of the exact value of each constant. In many steps of our proofs constants depending exclusively on the dimension of the ambient Lie group $G$ will appear. We will denote all these constants by $C(d_G)$ even though they might differ. We do however keep reference to the dimension to emphasise the dependence. The same convention applies to the lowercase $c(d_G)$. As a rule of thumb, $\alpha$ and $\epsilon$ will denote small quantities, and $\delta$ and $\rho$ will denote radii of certain balls. In \S~\ref{Section: Local Brunn--Minkowski via optimal transport}, $\lambda$ denotes Lebesgue measure in exponential coordinates.

\section{Local Brunn--Minkowski via optimal transport}\label{Section: Local Brunn--Minkowski via optimal transport}

\subsection{Preliminaries to the proof}
\subsubsection{Baker--Campbell--Hausdorff formula}
The BCH formula relates multiplication in a Lie group to addition and Lie brackets in its Lie algebra:

\begin{proposition}[\S 5.6, \cite{LieGroups15Hall}]\label{Proposition: BCH formula}
There is a neighbourhood $V$ of $0$ in $\mathfrak{g}$ such that for every $X,Y \in V$ we have 
\begin{equation}
    \log(\exp X\exp Y) = X + Y + \frac12 [X,Y] + \ldots, \label{Eq: BCH formula}
\end{equation}
where the remaining terms form a convergent series of iterated Lie brackets of total degree at least $3$.
\end{proposition}
\subsubsection{Haar measure in exponential coordinates}
In exponential coordinates, Haar measure has a smooth density whose first-order variation at the identity vanishes.

\begin{proposition}[Haar measure in exponential coordinates]\label{Proposition: Haar measure in exponential coordinates}
Let $G$ be a compact connected Lie group and $\mathfrak{g}$ its Lie algebra. There is a neighbourhood $U$ of $0$ in $\mathfrak{g}$ on which $\exp$ is a diffeomorphism onto its image, with inverse $\log$, and
\[ (\log)_*(\mu_G|_{\exp U})=\kappa\,\lambda|_U,\]
where $\lambda$ is a Lebesgue measure on $\mathfrak g$, $\kappa$ is smooth, $\kappa(0)>0$ and $\kappa(X)=\kappa(0)+O(\lVert X\rVert ^2)$.
\end{proposition}

\begin{proof}
The exponential map is a local diffeomorphism, so the density $\kappa$ is smooth and positive near $0$. Upon choosing the right normalisation of $\lambda$, we may assume $\kappa(0)=1$. Inversion preserves Haar measure in a compact group and corresponds to $X\mapsto -X$ in exponential coordinates. Hence $\kappa(X)=\kappa(-X)$, and the claimed estimate follows by Taylor expansion.
\end{proof}

With the normalisation $\kappa(0)=1$, one can be much more precise and show that
$\kappa(X)=\left|\det\frac{\sinh(\operatorname{ad}(X)/2)}{\operatorname{ad}(X)/2}\right|$, where the quotient is defined by its power series. We will not need the exact formula for $\kappa$ however.

\subsubsection{Optimal transport map}

We use the following result of Brenier: 

\begin{theorem}[Brenier, \cite{Brenier91}]\label{Theorem: Brenier}
Let $\mu_1, \mu_2$ be two compactly supported Borel probability measures on $\mathbb{R}^n$ that are absolutely continuous with respect to Lebesgue measure. There exists a convex function $f: \mathbb{R}^n \rightarrow \mathbb{R}$ such that $T=\nabla f$ satisfies $T_*\mu_1 = \mu_2$.
\end{theorem}

Note that the convex function $f$ considered above is not assumed to be differentiable everywhere. Convex functions are however differentiable almost everywhere, which is the meaning of $\nabla f$ above. We refer to \cite[Chap. 10]{VillaniOldNew09} for this and more. 
 
\subsubsection{Change of variables}

Our strategy relies on a general change of variables theorem that can be found in \cite{VillaniOldNew09}. It uses a notion of approximate differentiability and we refer to \cite[Def. 10.2]{VillaniOldNew09} for a definition.

\begin{theorem}[Change of variables, Thm. 11.1, \cite{VillaniOldNew09}]\label{Theorem: Change of variable}
Let $f \in L^1(\mathbb{R}^{d_G})$ be a non-negative function and let $T: \mathbb{R}^{d_G}\rightarrow \mathbb{R}^{d_G}$ be a Borel map. Define $d\mu(x)=f(x) d\lambda(x)$ and $\nu=T_*\mu$. Assume that: 
\begin{itemize}
    \item There exists a Borel set $\Sigma \subset \mathbb{R}^{d_G}$ such that $f=0$ almost everywhere outside of $\Sigma$ and $T$ is injective on $\Sigma$.
    \item $T$ is approximately differentiable almost everywhere on $\Sigma$.
\end{itemize}
Let $\tilde{\nabla} T$ be the approximate differential of $T$ and let $\mathcal{J}_T$ be defined almost everywhere on $\Sigma$ by $\mathcal{J}_T(x) := |\det \tilde{\nabla} T (x)|$. Then $\nu$ is absolutely continuous with respect to Lebesgue measure if and only if $\mathcal{J}_T > 0$ $\mu$-almost everywhere. In that case $\nu$ is concentrated on $T(\Sigma)$, and its density $g$ vanishes $\lambda$-almost everywhere outside $T(\Sigma)$ and satisfies, for $\mu$-almost every $x$, 
\begin{equation}
    f(x) = g(T(x)) \mathcal{J}_T(x). \label{Eq: Monge-Ampere}
\end{equation}
\end{theorem}
When $T$ is the gradient of a convex function, equation \eqref{Eq: Monge-Ampere} is the \emph{Monge--Amp\`ere equation}. It plays a key role in our approach.
\subsection{Proof of the local Brunn--Minkowski theorem}

Our strategy follows the one outlined in \cite[pp.~242--244]{FigalliICM14}. 

\begin{proof}[Proof of Theorem \ref{Theorem: Local Brunn--Minkowski}]
If either set has measure zero, a translate of the other set gives the result. We may therefore assume both measures are positive. Choose a ball $U\subset\mathfrak{g}$ about $0$ so small that the BCH product $x\cdot y:=\log(\exp x\exp y)$ is defined on $U\times U$. In this proof let $A,B\subset B_{\mathfrak{g}}(0,\rho)\subset U$ denote the logarithms of the original compact subsets. Let $\mu_A$ and $\mu_B$ be Lebesgue measure restricted to $A$ and $B$, normalised to mass $1$. By Brenier's theorem there is a convex function $f$ such that $T=\nabla f$ satisfies $T_*\mu_A=\mu_B$. We write $c(\rho)$ for non-negative quantities that may change from line to line and are $O_{G,d}(\rho)$, or $O_{d_G}(\rho)$ under our fixed semi-simple normalisation. We use $c(\rho)^2$ similarly for quadratic errors.

We work on a Borel subset $\Sigma\subset A$ of full Lebesgue measure in $A$ where $T$ is defined, takes values in $B$ and is injective. To justify the last condition, $T(x)\in\partial f(x)$ implies $x\in\partial f^*(T(x))$, where $f^*$ is the convex conjugate of $f$ and $\partial$ denotes the convex subdifferential. In particular, $T(x)$ lies in the domain of $f^*$. This domain is convex and has non-empty interior because $T_*\mu_A=\mu_B$ is absolutely continuous. Its boundary has Lebesgue measure zero, and $f^*$ is differentiable almost everywhere in its interior. We may therefore discard a null subset so that $\nabla f^*(T(x))=x$ on $\Sigma$, which gives injectivity. By Alexandrov's theorem, $T$ is approximately differentiable almost everywhere and $D_x:=\tilde{\nabla}T(x)$ is symmetric positive semidefinite \cite[Chap. 14, Appendix on Alexandrov's theorem]{VillaniOldNew09}. The change of variables formula now gives
\[\det D_x=\frac{\lambda(B)}{\lambda(A)}>0\]
almost everywhere on $\Sigma$, so $D_x$ is positive definite there.

\smallbreak
Define $F(x):=x\cdot T(x)$ on $\Sigma$. For $x,y\in\Sigma$, put $u=x-y$ and $v=T(x)-T(y)$. Convexity gives $\langle u,v\rangle\geq0$, hence
\[\lVert u+v\rVert \geq \lVert u\rVert ,\qquad \lVert u\rVert +\lVert v\rVert \leq\sqrt2\lVert u+v\rVert .\]
The map $(x,y)\mapsto x\cdot y-x-y$ has differential $O(\rho)$ on the coordinate ball. Consequently
\[\lVert F(x)-F(y)-(u+v)\rVert \leq c(\rho)(\lVert u\rVert +\lVert v\rVert )\leq c(\rho)\lVert u+v\rVert ,\]
and $\lVert F(x)-F(y)\rVert \geq(1-c(\rho))\lVert x-y\rVert $. Thus $F$ is injective for small enough $\rho$, and it is approximately differentiable almost everywhere on $\Sigma$.

\smallbreak
We now follow the optimal transport proof of the Euclidean Brunn--Minkowski inequality \cite{FigalliICM14} to bound the Jacobian of $F$. The BCH formula gives
\[\tilde{\nabla}F(x)=(\Id+A'_{\rho,x})+(\Id+A''_{\rho,x})D_x,\]
where
\[A'_{\rho,x}=-\tfrac12\operatorname{ad}(T(x))+O(\rho^2),\qquad A''_{\rho,x}=\tfrac12\operatorname{ad}(x)+O(\rho^2).\]
The operators $\operatorname{ad}(Y)$ are skew-symmetric for our invariant scalar product \cite[Prop. 4.24]{zbMATH01849081}. In particular $\det(\Id+A''_{\rho,x})\geq1-c(\rho)^2>0$. Factoring $\Id+A''_{\rho,x}$ out of $\tilde{\nabla}F(x)$ leaves $\Id+D_x+A_{\rho,x}$, where
\[\Id+A_{\rho,x}=(\Id+A''_{\rho,x})^{-1}(\Id+A'_{\rho,x})\]
and $A_{\rho,x}$ is the sum of a skew-symmetric operator and an operator of norm $O(\rho^2)$.

To obtain a bound independent of the size of the transport differential $D_x$, write $\Id+D_x+A_{\rho,x}=P+Q$, where $P$ is symmetric and $Q$ is skew-symmetric. Since $\Id+D_x\geq \Id$,
\[P\geq(1-c(\rho)^2)(\Id+D_x)>0.\]
The eigenvalues of the real skew-symmetric operator $P^{-1/2}QP^{-1/2}$ occur in pairs $\pm it$, with possible zeros. Therefore
\[\det(P+Q)=\det P\det(\Id+P^{-1/2}QP^{-1/2})\geq\det P\geq(1-c(\rho)^2)\det(\Id+D_x),\]
after absorbing the dimensional power into $c(\rho)^2$. If $\lambda_1,\ldots,\lambda_{d_G}>0$ are the eigenvalues of $D_x$, the AM-GM inequality \cite[Proof of Thm. 1]{zbMATH05649884} gives
\begin{align}
|\det\tilde{\nabla}F(x)|&\geq(1-c(\rho)^2)\left(1+\prod_{i=1}^{d_G}\lambda_i^{\frac1{d_G}}\right)^{d_G}\\
&\geq(1-c(\rho)^2)\left[1+\left(\frac{\lambda(B)}{\lambda(A)}\right)^{\frac1{d_G}}\right]^{d_G}.
\end{align}
\smallbreak
Applying the change of variables formula to $F$ on $\Sigma$, whose image lies in $A\cdot B$, yields
\[\lambda(A\cdot B)\geq\int_\Sigma|\det\tilde{\nabla}F(x)|\,d\lambda(x)\geq(1-c(\rho)^2)\left(\lambda(A)^{\frac1{d_G}}+\lambda(B)^{\frac1{d_G}}\right)^{d_G}.\]
Finally Proposition \ref{Proposition: Haar measure in exponential coordinates}, applied also on the product ball, gives
\[\mu_G(\exp A\exp B)\geq(1-c(\rho)^2)\left(\mu_G(\exp A)^{\frac1{d_G}}+\mu_G(\exp B)^{\frac1{d_G}}\right)^{d_G}.\]
Taking the $d_G$-th root and adjusting the constant proves the theorem.
\end{proof}

\section{Approximate subgroups of small measure}\label{Section: Approximate subgroups of small measure}

The goal of this section is to establish the inverse theorem (Theorem \ref{Proposition: approximate subgroups of small measure}) concerning the global structure of \emph{approximate subgroups} of small measure. 

We use throughout this section the shorthand $B_r=B_G(e,r)$ introduced in \S \ref{Subsection: Notation}, and we note that $\overline{B}_0=\{e\}$. 

\subsection{Preliminaries on approximate subgroups}

We will use throughout the following notion, due to Tao \cite{TaoProductSet08}.

\begin{definition*}
    Let $K \geq 1$. A subset $\Lambda$ of a group $G$ is a \emph{$K$-approximate subgroup} if $e \in \Lambda$, $\Lambda = \Lambda^{-1}$ and there is a finite subset $X \subset G$ with $|X| \leq K$ such that $$\Lambda^2 \subset X\Lambda.$$ An \emph{approximate subgroup} is a subset that is a $K$-approximate subgroup for some $K \geq 1$, and a \emph{compact approximate subgroup} is an approximate subgroup that is compact. 
\end{definition*}

Two consequences will be used repeatedly. First, $e \in \Lambda$ gives $\Lambda^i \subset \Lambda^j$ whenever $i \leq j$. Second, an induction on $m$ shows that $\Lambda^m \subset X^{m-1}\Lambda$ for all $m \geq 1$. Since $\Lambda^m$ is symmetric, taking inverses gives $\Lambda^m \subset \Lambda X^{-(m-1)}$ as well. So $\Lambda^m$ is covered by at most $K^{m-1}$ left translates of $\Lambda$, and by at most $K^{m-1}$ right translates of $\Lambda$. In particular, if $\Lambda^m$ is a $K^{2m - 1}$-approximate subgroup.

Approximate subgroups naturally appear in connection with \emph{small doubling}. A basic covering consequence is the following: 

\begin{lemma}[Ruzsa's covering lemma, Lem. 3.6, \cite{TaoProductSet08}]\label{Lemma: Covering lemma}
    Let $A, B \subset G$ be two compact subsets with $\mu_G(B)>0$, and let $K\geq1$. Suppose that $\mu_G(AB) \leq K \mu_G(B)$. Then there is a subset $F \subset G$ of size at most $K$ such that $A \subset FBB^{-1}$. 
\end{lemma}

We will rely on the following consequence of Tao's product-set theorem:

\begin{proposition}[Prop. 4.5 and Thm. 4.6, \cite{TaoProductSet08}]\label{Proposition: Tao}
Let $A,B\subset G$ be compact subsets of positive measure and put $K:=\mu_G(AB)/(\mu_G(A)^{1/2}\mu_G(B)^{1/2})$. There is a universal constant $C>0$ and a compact $(2K)^C$-approximate subgroup $\Lambda$ such that $A$ is covered by at most $(2K)^C$ left translates of $\Lambda$, $B$ is covered by at most $(2K)^C$ right translates of $\Lambda$, and
\[\mu_G(\Lambda)\leq(2K)^C\mu_G(A)^{\frac12}\mu_G(B)^{\frac12}.\]

\end{proposition}

\begin{proof}
Choose open neighbourhoods $U\supset A$ and $V\supset B$ with $\mu_G(U)\leq2\mu_G(A)$, $\mu_G(V)\leq2\mu_G(B)$ and $\mu_G(UV)\leq2\mu_G(AB)$. These exist by compactness, taking sufficiently small metric neighbourhoods. Theorem 4.6 of \cite{TaoProductSet08}, applied to $U,V$, gives an open $L$-approximate subgroup $Q$, the required left and right covers, and $\mu_G(Q)\leq L\mu_G(A)^{1/2}\mu_G(B)^{1/2}$, with $L=(2K)^{O(1)}$.

Set $\Lambda:=\overline Q$. Since $Q$ is open, symmetric and contains $e$, every $x\in\overline Q$ satisfies $xQ\cap Q\neq\emptyset$, so $\overline Q\subset Q^2$. If $Q^2\subset XQ$ with $|X|\leq L$, then $\Lambda^2\subset\overline{Q^2}\subset X\Lambda$. Thus $\Lambda$ is compact and $L$-approximate, its measure is at most $L\mu_G(Q)$, and the covers persist.

\end{proof}

\subsection{Overview of the proof}Before moving on, let us give a brief overview of the proof. We first prove the inverse theorem (Theorem~\ref{Proposition: approximate subgroups of small measure}) when $G$ is almost simple. We follow the shrinking commutator argument used by Breuillard and Green for finite approximate subgroups \cite{BGUnitary12}, recalled in \S~\ref{Subsection: Jordan argument}. The difficulty here is that $\Lambda^2$ need not have a shortest nonidentity element. Discretisation allows us to recover the argument whenever $\Lambda^2$ has a sufficiently wide gap in its distances from the identity.

For a fixed power of $\Lambda$, the approximate subgroup property supplies a controlled number of exceptional scales $r_1>\cdots>r_k$. Between consecutive scales, bounded products can be replaced by elements of $\Lambda^2$ with an error at the smaller scale. In other words, $\Lambda^2$ is (almost) closed under product in between these exceptional scales. This lets us take products repeatedly without continually increasing the number of factors.

Suppose two consecutive exceptional scales are far apart. If $\Lambda^2$ has a wide empty annulus between them, we choose a shortest element $\lambda\in\Lambda^2$ beyond the annulus and proceed with the shrinking commutator argument essentially as in the discrete case. This gives a substantial subset of $\Lambda$ whose quotients almost centralise $\lambda$, and from it we obtain the required covering near a proper connected subgroup.

If there is no such gap, one can find an element at each scale. Arguments from the theory of growth in groups show that each such element and its conjugates can be used to generate a ball of comparable radius in boundedly many products. In between these exceptional scales, this implies that $\Lambda^2$ already (approximately) contains this ball.  Passing through all the exceptional scales then covers a ball of fixed radius by a controlled number of translates of $\Lambda$, contradicting the small measure assumption. The covering number in this contradiction may depend on the requested radius $\delta$, whereas the covering near a subgroup must be independent of $\delta$.
\subsection{Discretising approximate subgroups}

``Discretising'' a subset $X \subset G$ means choosing a scale $r >0$ and considering the set $XB_r$.

Fix a power $\Lambda^m$. A covering of $\Lambda^m$ by finitely many translates of $\Lambda$ determines the scales at which new translates first meet a ball about $e$. The next lemma uses these scales to replace elements of $\Lambda^m$ by elements of $\Lambda^2$, with a controlled error.

\begin{lemma}[Boundedly many exceptional scales]\label{Lemma: Boundedly many interesting scales}
    Let $\Lambda$ be a compact symmetric subset of $G$ with $e \in \Lambda$ and let $m \geq 0$ be an integer. Let $F \subset G$ be a finite set such that $\Lambda^m \subset  \Lambda F$. Define 
    $r_f:=\inf_{\lambda \in \Lambda} d(e,\lambda f)$ for $f \in F$, let $k:=|\{r_f \vert f \in F\} \setminus \{0\}|$, and let $r_1> \ldots > r_k$ be an enumeration of $\{r_f \vert f \in F\} \setminus \{0\}$. For $s > 0$ set $r(s):=\max\left(\{r_f \vert f \in F,\ r_f < s\} \cup \{0\}\right)$. Then for all $s > 0$,
    \[ B_s \cap \Lambda^m \subset \Lambda^2\overline{B}_{r(s)}.\]
    In particular $r(r_i)=r_{i+1}$ for all $i=1,\ldots,k$ (with $r_{k+1}=0$), and $r(s)=r_1$ for all $s > r_1$ provided $k \geq 1$. 
\end{lemma}

\begin{proof}
For all $f \in F$, choose $\lambda_f\in\Lambda f$ minimising the distance to $e$, which exists by compactness, so that $d(e,\lambda_f)=r_f$. Let $s > 0$ and let $x \in B_s \cap \Lambda^m$. Since $\Lambda^m \subset \Lambda F$ we may write $x = \lambda f$ with $\lambda \in \Lambda$ and $f \in F$, and then $r_f \leq d(e,x) < s$. Writing $\lambda_f = \lambda' f$ with $\lambda' \in \Lambda$ and using that $\Lambda$ is symmetric, $x\lambda_f^{-1}=\lambda\lambda'^{-1} \in \Lambda^2$. Hence 
\[ B_s \cap \Lambda^m \subset \bigcup_{f \in F,\, r_f < s }\Lambda^2\lambda_f \subset \Lambda^2 \overline{B}_{r(s)},\]
the last inclusion because $d(e,\lambda_f)=r_f \leq r(s)$ for every $f$ occurring in the union.
\end{proof}

For consecutive exceptional scales, the conclusion of Lemma~\ref{Lemma: Boundedly many interesting scales} reads
\[\Lambda^m\cap B_{r_i}\subset\Lambda^2\overline B_{r_{i+1}}.\]
Thus every product of at most $m$ elements of $\Lambda$ that lies in $B_{r_i}$ is within distance $r_{i+1}$ of $\Lambda^2$. We will use this to replace products formed during the argument by elements of $\Lambda^2$, keeping track of the error through the exceptional scales. For a $K$-approximate subgroup and fixed $m\geq1$, there are at most $K^{m-1}$ such scales, by choosing a cover with that many translates.

Below the smallest exceptional scale, the error disappears. Positive measure then forces $\Lambda^2$ to contain the whole ball at that scale:

\begin{corollary}\label{Corollary: Smallest scale is doubled}
    Let $\Lambda$, $F$, $m$ and $r_1 > \ldots > r_k$ be as in Lemma \ref{Lemma: Boundedly many interesting scales}, and assume moreover that $m \geq 4$, that $\mu_G(\Lambda) > 0$ and that $k \geq 1$. Then
    \[B_{r_k} \subset \Lambda^2.\]
\end{corollary}

\begin{proof}
    Let $r:=\sup\{s \in (0,r_k] \vert B_s \subset \Lambda^2\}$. Since $\Lambda$ is symmetric and $\mu_G(\Lambda) > 0$, the subset $\Lambda^2$ contains a neighbourhood of $e$. To see this, $g\mapsto\mu_G(\Lambda\cap g\Lambda)$ is continuous by continuity of translations in $L^1$, and its value at $e$ is positive. Positivity implies $g\in\Lambda\Lambda^{-1}=\Lambda^2$. Thus the supremum above is taken over a non-empty set and $r > 0$. As $\Lambda$ is compact, so is $\Lambda^2$. Since $B_r$ is contained in the closure of $\bigcup_{s < r}B_s$, we get $B_r \subset \Lambda^2$. 

    We have $B_{2r} \subset B_rB_r \subset \Lambda^4 \subset \Lambda^m$, the first inclusion since $d$ is geodesic and the last because $e \in \Lambda$ and $m \geq 4$. Hence $B_{\min(2r,r_k)} \cap \Lambda^m = B_{\min(2r,r_k)}$ and Lemma \ref{Lemma: Boundedly many interesting scales} applied with $s=\min(2r,r_k)$ yields 
    \[B_{\min(2r,r_k)} \subset \Lambda^2\overline{B}_{r(s)} = \Lambda^2,\]
    since no $r_f$ is both non-zero and smaller than $r_k$, so that $r(s)=0$ and $\overline{B}_{r(s)}=\{e\}$.
    If $r<r_k$, then $\min(2r,r_k)>r$ is a larger admissible radius, a contradiction. Hence $r=r_k$.
\end{proof}

The next lemma shows how to extend an approximation from one scale to a larger one.

\begin{lemma}[Propagation between scales]\label{Lemma: Compounding all interests}
    Let $\Lambda \subset G$ be a symmetric subset with $e \in \Lambda$, and let $N \geq 2$ and $M \geq N+2$ be integers. Let $r_{\max} > r_{\min} > 0$. Suppose that $\Lambda^M \cap B_{r_{\max}} \subset \Lambda^2\overline{B}_{r_{\min}}$ and that there is $\frac{r_{\max}}{30}>r_0 \geq r_{\min}$ such that  $B_{5r_0} \subset \Lambda^NB_{r_0}$. Then \[B_{r_{\max} - 15r_0} \subset \Lambda^2\overline{B}_{r_{\min} + r_0}.\]
\end{lemma}
The inclusion $B_{5r_0}\subset\Lambda^N B_{r_0}$ gives an approximation to a ball at scale $r_0$. We want to extend it to a ball whose radius is close to $r_{\max}$, while keeping the error at most $r_{\min}+r_0$. Multiplication enlarges the ball, but also increases the power of $\Lambda$ and the approximation error. The hypothesis on $\Lambda^M\cap B_{r_{\max}}$ lets us return to $\Lambda^2$, and the approximation at scale $r_0$ lets us reduce the error again. We repeat this enlargement until we reach the boundary allowed by $r_{\max}$.

\begin{proof}
Let $r'\leq r_{\max}-15r_0$ be the largest radius such that
\[B_{r'}\subset\Lambda^2\overline B_{r_{\min}+r_0}.\]
Such a radius exists. Indeed $e\in\Lambda^2$ makes $r_{\min}+r_0$ admissible, since $r_{\min}+16r_0\leq17r_0<r_{\max}$. The supremum of the admissible radii is again admissible because $B_r=\bigcup_{s<r}B_s$.

\smallbreak
We enlarge this ball using $B_{5r_0}\subset\Lambda^N B_{r_0}$. Since the metric is geodesic and the balls are invariant under conjugation,
\[B_{r'+5r_0}\subset B_{r'}B_{5r_0}
\subset\Lambda^{N+2}\overline B_{r_{\min}+2r_0}.\]
The triangle inequality restricts the product in $\Lambda^{N+2}$ to $B_{r'+r_{\min}+7r_0}\subset B_{r_{\max}}$. As $N+2\leq M$, the hypothesis on $\Lambda^M$ therefore gives
\[B_{r'+5r_0}\subset\Lambda^2\overline B_{2r_{\min}+2r_0}.\]

\smallbreak
We now reduce the error to its original size. Since $r_{\min}\leq r_0$,
\[\overline B_{2r_{\min}+2r_0}\subset B_{5r_0}\subset\Lambda^N B_{r_0}.\]
Consequently $B_{r'+5r_0}\subset\Lambda^{N+2}B_{r_0}$. For each point of $B_{r'+5r_0}$, the triangle inequality bounds the distance of the corresponding element of $\Lambda^{N+2}$ from $e$ by $r'+6r_0$. Thus
\begin{align*}
B_{r'+5r_0}
&\subset\bigl(\Lambda^{N+2}\cap B_{r'+6r_0}\bigr)B_{r_0}\\
&\subset\Lambda^2\overline B_{r_{\min}}B_{r_0}
\subset\Lambda^2\overline B_{r_{\min}+r_0}.
\end{align*}
The second inclusion again uses $N+2\leq M$ and $r'+6r_0<r_{\max}$. If $r'<r_{\max}-15r_0$, this makes $\min(r'+5r_0,r_{\max}-15r_0)$ a larger admissible radius, contradicting maximality. Thus $r'=r_{\max}-15r_0$, as required.
\end{proof}

To apply the propagation lemma (Lemma~\ref{Lemma: Compounding all interests}), we need an approximation to a ball at one scale. For almost simple $G$, we will obtain it by conjugating an element near the identity, provided $\Lambda$ avoids neighbourhoods of proper connected subgroups. The generation lemma (Lemma~\ref{Lemma: Generation}) supplies the corresponding statement for sums in an irreducible representation. Its proof uses de Saxc\'e's escape estimates in the adjoint group $G/Z(G)$. The following observation will ensure that passing to this quotient preserves the required subgroup-avoidance condition. 

\begin{lemma}\label{Lemma: Central elements are far from maximal subgroups}
    Let $G$ be a compact connected semi-simple Lie group. There is $\varepsilon_Z > 0$, depending on $G$ only, such that for every $z$ in the centre $Z(G)$ and every maximal proper closed connected subgroup $M \subset G$, if $d(z,M) < \varepsilon_Z$ then $z \in M$.
\end{lemma}

\begin{proof}
As $G$ is semi-simple its centre $Z(G)$ is finite, and by \cite[\S 2 and \S 4.7]{BurnessLiebeckShalev2020} there are finitely many conjugacy classes of maximal proper closed connected subgroups of $G$. Let $M_1, \ldots, M_N$ be representatives. Set
\[\varepsilon_Z := \min\left(\{1\} \cup \left\{d(z,M_i) \vert z \in Z(G),\ 1 \leq i \leq N,\ z \notin M_i\right\}\right),\]
a minimum of finitely many strictly positive numbers (each $M_i$ being closed), hence strictly positive, and depending on $G$ only.

Now let $z \in Z(G)$ and let $M$ be a maximal proper closed connected subgroup with $d(z,M) < \varepsilon_Z$. Write $M = gM_ig^{-1}$ for some $g \in G$ and some $i$. Since $d$ is bi-invariant and $z$ is central, $d(z,M)=d(g^{-1}zg, M_i)=d(z,M_i)$, and likewise $z \in M$ if and only if $z \in M_i$. So $d(z,M_i) < \varepsilon_Z$, which by the choice of $\varepsilon_Z$ forces $z \in M_i$, hence $z \in M$.
\end{proof}

\begin{lemma}[Generation]\label{Lemma: Generation}
    Let $G$ be a compact connected almost simple Lie group, and let $V$ be an irreducible real representation of $G$ of dimension $d_V$ equipped with an invariant norm, and suppose that the centre $Z(G)$ acts trivially on $V$. There are $C \geq 1$, depending on $d_V$ only, and $\delta_0 \in \left(0,\frac12\right)$ and $r_V>0$, depending on $G$ and $V$ only, such that the following holds for every $\delta \in (0,\delta_0)$ and every integer $m \geq 1$. If $X \subset B_{r_V}$ is a symmetric subset with $e \in X$ that is not contained in the $\delta$-neighbourhood of any proper closed connected subgroup of $G$, and $v \in V$ is a unit vector, then \[B_V(0,\delta^Cm)\subset \left(\sum_{d_V}X^{d_V}\cdot([m]v) \right) + B_V(0,d_V),\]
    where $[m]v$ denotes $\{-mv,\ldots,mv\}$, $\cdot$ denotes the action of $G$, and $\sum_{d_V}Y$ denotes the $d_V$-fold sumset $Y + \cdots + Y$ of a subset $Y \subset V$.
\end{lemma}

\begin{proof}

If $V$ is the trivial representation, then $d_V=1$ and rounding a real coefficient to an integer proves the assertion. Assume from now on that $V$ is non-trivial. We want to find a basis of $V$ among the vectors in $X^{d_V}\cdot v$, with each vector separated from the span of its predecessors by a power of $\delta$. Such a basis allows us to express every vector in a ball using coefficients bounded by $m$. Rounding those coefficients then gives the integer combinations in the conclusion.

We construct the basis by applying de Saxc\'e's escape estimates successively to the spans already obtained. These estimates apply in the adjoint group $G/Z(G)$, so we first check that the image of $X$ still avoids the $\delta$-neighbourhoods of proper connected subgroups.

\smallbreak
Let $Z:=Z(G)$, which is finite as $G$ is semi-simple, let $\overline{G}:=G/Z$ and let $\pi: G \rightarrow \overline{G}$ be the projection. Right translations by elements of $Z$ are isometries of $(G,d)$, so $\overline{G}$ carries the quotient Riemannian metric $\overline{d}$ and $\overline{d}(\pi(x),\pi(S))=d(x,SZ)$ for every $x \in G$ and every $S \subset G$. The group $\overline{G}$ is compact almost simple with trivial centre, and $V$ is a representation of $\overline{G}$ by assumption, the action of $\pi(g)$ being that of $g$. Let $\varepsilon_Z>0$ be given by Lemma \ref{Lemma: Central elements are far from maximal subgroups} and take $r_V \leq \frac{\varepsilon_Z}{3}$ and $\delta_0 \leq \frac{\varepsilon_Z}{3}$. We shrink both further below.

We claim that $\pi(X)$ is not contained in the $\delta$-neighbourhood of any proper closed connected subgroup of $\overline{G}$. Suppose it were, say $\pi(X) \subset \overline{H}_{\delta}$ with $\overline{H} \subset \overline{G}$ proper closed connected, and set $L:=\pi^{-1}(\overline{H})$, a proper closed subgroup of $G$ satisfying $LZ=L$. Then $X \subset L_{\delta}$ by the displayed identity above. The restriction of $\pi$ to $L^{\circ}$ has open image in $\overline{H}$ and $\overline{H}$ is connected, so $\pi(L^{\circ})=\overline{H}$ and therefore $L=L^{\circ}Z$. Among the proper closed connected subgroups of $G$ containing $L^{\circ}$ choose one, say $M$, of maximal dimension. It is a maximal proper closed connected subgroup, since a connected subgroup containing $M$ and of the same dimension as $M$ has the same Lie algebra as $M$, hence equals $M$.

We show that every component of $L$ meeting the $\delta$-neighbourhood of $X$ is contained in $M$. Let $z \in Z$ be such that $L^{\circ}z$ meets the $\delta$-neighbourhood of some $x \in X$. As $d(e,x)<r_V$ we get $d(e,L^{\circ}z)<r_V+\delta$, that is $d(z^{-1},L^{\circ})<r_V+\delta$, whence $d(z^{-1},M)<r_V+\delta<\varepsilon_Z$ and $z^{-1} \in M$ by Lemma \ref{Lemma: Central elements are far from maximal subgroups}. Consequently $L^{\circ}z \subset M$ for every such $z$, and $X \subset M_{\delta}$, contradicting the hypothesis on $X$.

\smallbreak
Choose an invariant Euclidean norm on $V$. For the given invariant norm, the uniqueness of the John ellipsoid gives such a Euclidean norm with comparison constants at most $\sqrt{d_V}$. These constants will be absorbed below. Apply \cite[Prop. 2.7, Prop. 2.8 and Rem. 1]{zbMATH06466329} to $\overline G$, shrinking $r_V$ and $\delta_0$ as required. Although Proposition 2.8 is stated for complex representations, it does not require irreducibility. We apply it to $V\otimes_{\mathbb R}\mathbb C$ and the complexifications of real subspaces. For a real vector, distance to the complexified subspace is its distance to the real subspace. The compactness argument of Remark 1 applies to the real unit sphere and real Grassmannians because $V$ is irreducible over $\mathbb R$. It therefore gives the required uniform escape constant even if the complexification is reducible.

Repeatedly escaping the span of the preceding vectors gives $\lambda_1=e,\ldots,\lambda_{d_V}\in X^{d_V}$ such that
\[\operatorname{dist}(\lambda_i\cdot v,W_i)\geq\delta^{C'},\qquad W_i:=\operatorname{span}(\lambda_1\cdot v,\ldots,\lambda_{i-1}\cdot v).\]
Here $C'$ may be chosen to depend only on $d_V$. Indeed the non-trivial representation of the centreless almost simple group $\overline G$ is faithful, so $\dim G\leq d_V^2$. There are finitely many possible almost simple groups in this dimension range and, by the Weyl dimension formula, finitely many irreducible real representations of each of dimension $d_V$. Taking the maximum of the finitely many exponents gives the assertion. Fixed comparison constants for norms are absorbed by increasing $C'$ and shrinking $\delta_0$.

\smallbreak
The almost orthogonal basis estimate \cite[Lem. 2.16]{zbMATH06466329} now shows, after increasing $C=C(d_V)$, that the map
\begin{align*}
\phi:\mathbb R^{d_V}&\longrightarrow V\\
(t_1,\ldots,t_{d_V})&\longmapsto\sum_i t_i(\lambda_i\cdot v)
\end{align*}
satisfies $B_V(0,\delta^Cm)\subset\phi(B_{\mathbb R^{d_V}}(0,m))$. Round each coordinate of a point in this Euclidean ball to a nearest integer in $\{-m,\ldots,m\}$. Each coordinate changes by at most $1/2$, so its image under $\phi$ changes by at most $d_V/2$ in the original invariant norm. Hence
\[B_V(0,\delta^Cm)\subset\phi(\{-m,\ldots,m\}^{d_V})+B_V(0,d_V)\subset\left(\sum_{d_V}X^{d_V}\cdot([m]v)\right)+B_V(0,d_V).\]
\end{proof}

When applied to the adjoint representation, the generation lemma (Lemma~\ref{Lemma: Generation}) produces sums of conjugates of a vector $X\in\mathfrak g$. We need to turn these sums into products in $G$. If $\exp(X)\in\Lambda^2$, exponentiating each summand gives an element of a controlled power of $\Lambda$. The BCH formula (Proposition~\ref{Proposition: BCH formula}) compares their product with the exponential of their sum. The approximation principle (Lemma~\ref{Lemma: Approximation principle}) chooses the radius $r_a$ so that the accumulated error remains comparable to $\lVert X\rVert$.

For $X,Y\in\mathfrak g$ with $\lVert X\rVert,\lVert Y\rVert\leq\rho$, the BCH formula gives
\[d\left(\exp(X)\exp(Y),\exp(X+Y)\right)=O_{d_G}(\rho^2).\]
Iterating this estimate yields the following statement:

\begin{lemma}[Approximation principle]\label{Lemma: Approximation principle}
    Let $m \geq 4d_G$ be an integer. There is $r_a > 0$ depending on $d_G$ and $m$ only such that the following holds. Let $\Lambda \subset G$ be a symmetric subset with $e \in \Lambda$ and let $X \in B_{\mathfrak{g}}(0,2r_a)$ be such that $\exp(X) \in \Lambda^2$. Then 
    \[\exp\left(\sum_{d_G}\Lambda^{d_G}\cdot([m]X) + B_{\mathfrak{g}}(0,d_G\lVert X\rVert )\right) \subset \Lambda^{d_G^2m}B_{2d_G\lVert X\rVert },\]
    where $\cdot$ denotes the adjoint action and $[m]X$ denotes $\{-mX, \ldots, mX\}$. One may take $r_a=c(d_G)m^{-2}$ with $c(d_G)>0$. 
\end{lemma}

\begin{proof}
We take $2r_a$ smaller than the injectivity radius of $G$ at $e$, so that $\lVert X\rVert =d(e,\exp(X)) < 2r_a$, $d$ being geodesic. We shrink $r_a$ further below. 

Take $Y_1, \ldots, Y_{d_G}$ in $\Lambda^{d_G}\cdot([m]X)$ and $Z \in B_{\mathfrak{g}}(0,d_G\lVert X\rVert )$. For each $i$, write $Y_i = \Ad(\lambda_i)(k_iX)$ with $\lambda_i \in \Lambda^{d_G}$ and $|k_i| \leq m$. Since $\exp(\Ad(\lambda)Y)=\lambda\exp(Y)\lambda^{-1}$ and $\exp(k_iX)=\exp(X)^{k_i}$, and since $\Lambda$ is symmetric, we get $\exp(X)^{k_i} \in \Lambda^{2m}$ and $\exp(Y_i) \in \Lambda^{d_G}\Lambda^{2m}\Lambda^{d_G}=\Lambda^{2d_G+2m}$. Hence 
\[ \prod_{i=1}^{d_G}\exp(Y_i) \in \Lambda^{d_G(2d_G+2m)} \subset \Lambda^{d_G^2m},\]
the last inclusion because $e \in \Lambda$ and $2d_G + 2m \leq d_Gm$, the latter holding since $m \geq 4d_G$ and $d_G \geq 3$. Indeed, a compact connected Lie group of dimension at most $2$ is a torus, hence abelian, so a semi-simple $G$ has $d_G \geq 3$. 

The scalar product $\langle \cdot, \cdot\rangle$ being $\Ad(G)$-invariant, $\lVert Y_i\rVert =|k_i|\,\lVert X\rVert  \leq m\lVert X\rVert $, and $\lVert Z\rVert  \leq d_G\lVert X\rVert  \leq m\lVert X\rVert $. Shrink $r_a$ so that $2(d_G+1)mr_a$ lies below the fixed BCH coordinate radius. The partial sums of the $Y_i$ and $Z$ then stay in this coordinate neighbourhood. Repeated applications of the estimate above provide a constant $C_5>0$ depending on $d_G$ only such that 
\[ d\left(\exp\left(\sum_iY_i + Z\right), \prod_i\exp(Y_i)\exp(Z)\right) \leq C_5m^2\lVert X\rVert ^2.\]
Since $d(e,\exp(Z)) \leq \lVert Z\rVert  \leq d_G\lVert X\rVert $, we conclude that 
\[ \exp\left(\sum_iY_i+Z\right) \in \Lambda^{d_G^2m}B_{d_G\lVert X\rVert +C_5m^2\lVert X\rVert ^2}.\]
Choosing $r_a$ so that moreover $2C_5m^2r_a \leq d_G$, we have $C_5m^2\lVert X\rVert ^2 \leq d_G\lVert X\rVert $. All the restrictions on $r_a$ are satisfied by $r_a=c(d_G)m^{-2}$ for sufficiently small $c(d_G)>0$, which concludes. 
\end{proof}

\subsection{Jordan's theorem for discretised sets}\label{Subsection: Jordan argument}
We recall the shrinking commutator argument before applying it to the almost simple case. With the convention $[g,h]:=ghg^{-1}h^{-1}$, there is $0<r_c<1$ such that, for $g,h\in B_{r_c}$ with $g\neq e$,
\[d(e,[g,h])<\frac{d(e,g)}{2}.\]

Consider first a finite subgroup $\Gamma$ of $G$. Since $G$ is semi-simple, we may shrink $r_c$ so that $B_{r_c}\cap Z(G)=\{e\}$. If $\Gamma\cap B_{r_c}=\{e\}$, a packing argument bounds $|\Gamma|$ in terms of $G$. Otherwise, choose $\gamma\in\Gamma\setminus\{e\}$ minimising $d(e,\gamma)$. Then $\gamma\in B_{r_c}$, and for every $x\in\Gamma\cap B_{r_c}$,
\[d(e,[\gamma,x])<\frac12d(e,\gamma).\]
But $[\gamma,x]\in\Gamma$, so minimality forces $[\gamma,x]=e$. Consequently,
\[\Gamma\cap B_{r_c}\subset C_G(\gamma),\]
where $C_G(\gamma):=\{g\in G\vert g\gamma=\gamma g\}$ is proper because $\gamma$ is not central. A packing argument covers $\Gamma$ by boundedly many translates of $\Gamma\cap B_{r_c}$. In other words, $\Gamma$ is covered by boundedly many cosets of the \emph{proper} subgroup $C_G(\gamma)$ of $G$.

For finite approximate subgroups, commutators belong to a larger power of the set, so this minimality argument no longer forces them to vanish. Breuillard and Green instead count the possible commutators and find a large subset whose quotients centralise $\gamma$ \cite[\S\S~2--3]{BGUnitary12}.

For our compact approximate subgroup $\Lambda$, a gap between radii $r$ and $\eta r$ allows us to carry out this counting argument at resolution $r$. We choose a shortest element $\lambda\in\Lambda^2$ beyond the gap and find a substantial subset $\Lambda'\subset\Lambda$ whose commutators with $\lambda$ lie in one ball of radius $r$. The quotients in $\Lambda'^{-1}\Lambda'$ then almost centralise $\lambda$, with an error small relative to $d(e,\lambda)$. The generation lemma (Lemma~\ref{Lemma: Generation}) turns this almost centralisation into containment near a proper connected subgroup.

We first isolate the consequence of an empty annulus. The covering bound in the next lemma is independent of $\delta$, although the required width of the annulus depends on $\delta$.

\begin{lemma}[Shrinking commutator and an empty annulus]\label{Lemma: Gap gives subgroup covering}
Let $G$ be a compact connected almost simple Lie group with the fixed metric. There are $R>0$, $\delta_*>0$ and $C\geq1$, depending only on $d_G$, with the following property. Let $0<\delta<\delta_*$, let $\eta\geq\delta^{-C-1}$, and let $\Lambda\subset\overline B_R$ be a compact $K$-approximate subgroup of positive measure. If, for some $0<r\leq\delta$,
\[
\Lambda^2\cap B_{\eta r}\subset B_r,
\]
then there are a proper closed connected subgroup $H\subset G$ and a finite set $F\subset G$ with $|F|\leq K^7$ such that $\Lambda\subset FH_\delta$.
\end{lemma}

\begin{proof}
Let $C\geq1$, $\delta_0>0$ and $r_V>0$ be given by Lemma~\ref{Lemma: Generation} for the adjoint representation. Fix $r_L>0$ with $2r_L$ below the injectivity radius, and a Lipschitz constant $C_G\geq1$ for $\log$ on $B_{2r_L}$. Choose $R>0$ so that
\[
\overline B_{2R}\subset B_{r_c}\cap B_{r_V}\cap B_{2r_L},
\]
and take
\[
\delta_*\leq\min\left(\frac{\delta_0}{2},\frac{9}{40C_Gd_G^2},\frac12\right).
\]
These choices depend only on $d_G$. Set $m:=\lceil10d_G\delta^{-C}\rceil$.

If $\Lambda^2\subset B_r$, then $\Lambda\subset B_\delta$ and the trivial subgroup gives the conclusion. Otherwise choose $\lambda\in\Lambda^2\setminus B_r$ minimising $d(e,\lambda)$. Compactness gives such an element, and the empty annulus gives $d(e,\lambda)\geq\eta r$. We also have $d(e,\lambda)<r_c$ by the choice of $R$.

As $\Lambda^2 \subset B_{r_c}$, for all $\lambda' \in \Lambda$, 
\[d(e,[\lambda,\lambda']) < \frac{d(e,\lambda)}{2}.\]
But for any two $f_1, f_2 \in \Lambda^6 \cap B_{\frac{d(e,\lambda)}{2}}$ we have $f_1\Lambda \cap f_2\Lambda \neq \emptyset$ only if $f_1^{-1}f_2 \in B_{r}$. Indeed, if $f_1\Lambda \cap f_2 \Lambda \neq \emptyset$, then $f_1^{-1}f_2 \in \Lambda^2 \cap B_{d(e,\lambda)} \subset B_{r}$. If $E \subset \Lambda^6 \cap B_{\frac{d(e,\lambda)}{2}}$ is a maximal $r$-separated subset (i.e. any two distinct points are at distance at least $r$), then distinct $f_1, f_2 \in E$ satisfy $f_1^{-1}f_2 \notin B_{r}$, so the translates $(f\Lambda)_{f \in E}$ are pairwise disjoint by the above. By left-invariance of the Haar measure, and since $E \subset \Lambda^6$, \[|E|\mu_G(\Lambda)=\mu_G(E \Lambda) \leq \mu_G(\Lambda^7) \leq K^6\mu_G(\Lambda)\]
i.e. $|E| \leq K^6$. As $E$ is a maximal $r$-separated subset it is an $r$-net, and the commutators $[\lambda, \lambda']$ for $\lambda' \in \Lambda$ all lie in $\Lambda^6 \cap B_{\frac{d(e,\lambda)}{2}}$, so the subsets $\{\lambda' \in \Lambda \vert d([\lambda,\lambda'],f) \leq r\}$, for $f \in E$, cover $\Lambda$. Each of them is measurable, being the preimage of a closed ball under the continuous map $\lambda' \mapsto [\lambda,\lambda']$. There are thus $f \in E$ and a compact subset $\Lambda' \subset \Lambda$ of measure at least $\frac{\mu_G(\Lambda)}{K^6}$ such that $d([\lambda,\lambda'],f) \leq r$ for all $\lambda' \in \Lambda'$. 

Write $\Xi:=\Lambda'^{-1}\Lambda'$, a symmetric subset containing $e$ and contained in $\Lambda^2 \subset \overline B_{2R} \subset B_{r_V}$. We claim that every $w \in \Xi$ almost commutes with $\lambda$, in the sense that 
\begin{equation}
    d(\lambda, w\lambda w^{-1}) \leq 2r. \label{Eq: Almost centralising}
\end{equation}
Indeed, let $\lambda_1', \lambda_2' \in \Lambda'$ and $w:=\lambda_1'^{-1}\lambda_2'$. The two commutators are within $2r$ of each other, that is $d(\lambda\lambda_1'\lambda^{-1}\lambda_1'^{-1}, \lambda\lambda_2'\lambda^{-1}\lambda_2'^{-1}) \leq 2r$. Cancelling $\lambda$ on the left and conjugating by $\lambda_1'^{-1}$, both of which preserve $d$ since it is bi-invariant, this reads $d(\lambda^{-1}, w\lambda^{-1}w^{-1}) \leq 2r$, and inversion is an isometry of $(G,d)$, whence \eqref{Eq: Almost centralising}. 

\smallbreak
We now use the gap to show that $\Xi=\Lambda'^{-1}\Lambda'$ lies near a proper connected subgroup. By our fixed choice of $R$, $\lambda\in\overline B_{2R}\subset B_{2r_L}$. Thus $X:=\log\lambda$ is defined and $\lVert X\rVert=d(e,\lambda)$. Put $u:=X/\lVert X\rVert$. The elements of $\Xi$ almost centralise $\lambda$, and the gap makes their error small relative to $\lVert X\rVert$. Their adjoint actions therefore almost fix $u$. The sums constructed from this orbit remain close to the line $\mathbb Ru$, whereas the generation lemma (Lemma~\ref{Lemma: Generation}) would make them approximate a ball unless $\Xi$ lies near a proper connected subgroup. We now make this precise.

Since $w\lambda w^{-1}=\exp(\Ad(w)X)$ and $w\lambda w^{-1}\in B_{2r_L}$, the Lipschitz bound for $\log$ and \eqref{Eq: Almost centralising} give $\lVert\Ad(w)X-X\rVert\leq2C_Gr$. Using $d(e,\lambda)\geq\eta r$, we obtain
\[\lVert\Ad(w)u-u\rVert\leq\frac{2C_Gr}{d(e,\lambda)}\leq\frac{2C_G}{\eta}\qquad(w\in\Xi).\]
The adjoint action is by isometries of $\mathfrak{g}$, so these errors add along products: for $w=w_1\cdots w_{d_G} \in \Xi^{d_G}$ we get $\lVert \Ad(w)u-u\rVert  \leq \frac{2C_Gd_G}{\eta}$, by writing $\Ad(w)u-u$ as the telescoping sum of the $\Ad(w_1\cdots w_{j-1})\left(\Ad(w_j)u-u\right)$. An element of $\sum_{d_G}\Xi^{d_G}\cdot\left([m]u\right)$ is of the form $\sum_{j=1}^{d_G}k_j\Ad(w_j)u$ with $|k_j| \leq m$, hence lies within $\frac{2C_Gd_G^2m}{\eta}$ of $\sum_jk_ju \in \mathbb{R}u$. Therefore 
\[ \sum_{d_G}\Xi^{d_G}\cdot\left([m]u\right) + B_{\mathfrak{g}}(0,d_G) \subset \mathbb{R}u + B_{\mathfrak{g}}\left(0,\frac{2C_Gd_G^2m}{\eta}+d_G\right).\]

Suppose now that $\Xi$ is not contained in the $\delta$-neighbourhood of any proper closed connected subgroup of $G$. Then Lemma \ref{Lemma: Generation}, applied to the adjoint representation, to $\Xi$ and to the unit vector $u$, gives $B_{\mathfrak{g}}(0,\delta^Cm) \subset \sum_{d_G}\Xi^{d_G}\cdot\left([m]u\right) + B_{\mathfrak{g}}(0,d_G)$. Projecting the two inclusions onto the non-zero space $u^\perp$ shows that
\[ \delta^Cm \leq \frac{2C_Gd_G^2m}{\eta}+d_G.\]
But $m \geq 10d_G\delta^{-C}$ gives $d_G \leq \frac{\delta^Cm}{10}$, so this forces $\frac{9}{10}\delta^C \leq \frac{2C_Gd_G^2}{\eta}$. Since $\eta\geq\delta^{-C-1}$, the right-hand side is at most $2C_Gd_G^2\delta^{C+1}$, and the inequality fails as soon as $\delta < \frac{9}{20C_Gd_G^2}$, as ensured by our initial choice of $\delta$. 

So there is a proper closed connected subgroup $H \subset G$ with $\Xi \subset H_{\delta}$. Finally, $\mu_G(\Lambda\Lambda'^{-1}) \leq \mu_G(\Lambda^2) \leq K\mu_G(\Lambda) \leq K^7\mu_G(\Lambda')=K^7\mu_G(\Lambda'^{-1})$, the last equality because $\mu_G$ is inverse-invariant, $G$ being compact. Ruzsa's covering lemma (Lemma \ref{Lemma: Covering lemma}) applied to $\Lambda$ and $\Lambda'^{-1}$ therefore provides $F \subset G$ with $|F| \leq K^7$ and $\Lambda \subset F\Lambda'^{-1}\Lambda'=F\Xi \subset FH_{\delta}$. 
\end{proof}

The complementary argument starts from an element in the annulus. If $\Lambda$ avoids neighbourhoods of proper connected subgroups, this element and its conjugates by $\Lambda$ generate an approximation to a ball. The following lemma combines this construction with the propagation lemma, for two arbitrary radii.

\begin{lemma}[Generating a ball between two scales]\label{Lemma: Ball between two scales}
Let $G$ be a compact connected almost simple Lie group with the fixed metric. There are $R>0$, $\delta_*>0$, $c>0$ and $C\geq1$, depending only on $d_G$, with the following property. For $0<\delta<\delta_*$, put
\[
m:=\lceil10d_G\delta^{-C}\rceil,\qquad M:=d_G^2m+2,\qquad r_a:=cm^{-2}.
\]
Let $\Lambda\subset\overline B_R$ be symmetric, contain $e$, and not be contained in $H_\delta$ for any proper closed connected subgroup $H$. Suppose that $\eta\geq1$ and
\[
0<r_{\min}<r_{\max}\leq r_a,\qquad 60d_G\eta r_{\min}<r_{\max},
\]
and that
\[
\Lambda^M\cap B_{r_{\max}}\subset\Lambda^2\overline B_{r_{\min}},\qquad
\Lambda^2\cap(B_{\eta r_{\min}}\setminus B_{r_{\min}})\neq\emptyset.
\]
Then
\[
B_{r_{\max}-30d_G\eta r_{\min}}\subset\Lambda^2\overline B_{3d_G\eta r_{\min}}.
\]
\end{lemma}

\begin{proof}
Choose $C$, $\delta_*$ and $R$ so that Lemma~\ref{Lemma: Generation} applies to $\Lambda$ in the adjoint representation, and take $\delta_*<1/2$. The choice of $m$ gives $\delta^Cm\geq10d_G$ and $m\geq4d_G$. By Lemma~\ref{Lemma: Approximation principle}, we may take $r_a=cm^{-2}$ for a sufficiently small dimensional constant $c>0$. Decrease $c$ so that $r_a\leq R$ and $2r_a$ is below the injectivity radius.

Choose $\lambda\in\Lambda^2\cap(B_{\eta r_{\min}}\setminus B_{r_{\min}})$ and put $\ell:=d(e,\lambda)$. Then
\[
r_{\min}\leq\ell<\eta r_{\min}<\frac{r_{\max}}{60d_G}\leq\frac{r_a}{60d_G}.
\]
In particular $\log\lambda$ is defined and has norm $\ell$. Applying Lemma~\ref{Lemma: Generation} to the unit vector $\log\lambda/\ell$, and rescaling by $\ell$, gives 
\[
B_{\mathfrak g}(0,10d_G\ell)
\subset\left(\sum_{d_G}\Lambda^{d_G}\cdot([m]\log\lambda)\right)
+B_{\mathfrak g}(0,d_G\ell).
\]
The approximation principle (Lemma \ref{Lemma: Approximation principle})  now applies because $\ell<2r_a$. Since $10d_G\ell$ is below the injectivity radius, exponentiating yields
\[
B_{10d_G\ell}\subset\Lambda^{d_G^2m}B_{2d_G\ell}.
\]

\smallbreak
We extend this approximation towards $r_{\max}$ using Lemma~\ref{Lemma: Compounding all interests}. Set $s:=2d_G\ell$ and $N:=d_G^2m$, so that $M=N+2$. The preceding inclusion is $B_{5s}\subset\Lambda^N B_s$, and
\[
r_{\min}\leq s<2d_G\eta r_{\min},\qquad 30s<r_{\max}.
\]
Together with the assumed inclusion for $\Lambda^M\cap B_{r_{\max}}$, these are the hypotheses of Lemma~\ref{Lemma: Compounding all interests}, with $s$ in the role of its $r_0$. It follows that
\[
B_{r_{\max}-15s}\subset\Lambda^2\overline B_{r_{\min}+s}.
\]
Now $15s<30d_G\eta r_{\min}$ and $r_{\min}+s\leq3d_G\eta r_{\min}$, which give the required inclusion.
\end{proof}

We now combine the two lemmas through the exceptional scales to prove the inverse theorem for almost simple groups.

\begin{proposition}\label{Proposition: approximate subgroups of small measure, simple case}
    Let $G$ be a compact connected almost simple Lie group of dimension $d_G$, let $K \geq 1$ and let $\delta > 0$. There is $\epsilon > 0$ such that for all compact $K$-approximate subgroups $\Lambda \subset G$ with $e \in \Lambda$, $\mu_G(\Lambda) > 0$ and $\mu_G(\Lambda) \leq \epsilon$ there is a proper closed connected subgroup $H$ such that $\Lambda$ is covered by $C(d_G)K^{100}$ translates of  $H_{\delta}$.
\end{proposition}

\begin{proof}
Let us briefly summarise the proof. We first localise $\Lambda$ and apply Lemma~\ref{Lemma: Boundedly many interesting scales}. An empty annulus across a sufficiently wide interval gives the desired subgroup covering by Lemma~\ref{Lemma: Gap gives subgroup covering}. Otherwise Lemma~\ref{Lemma: Ball between two scales} lets us extend covers from each exceptional scale to the next larger one. Starting from the ball contained in $\Lambda^2$ at the smallest scale, we obtain a cover of a fixed ball and hence a positive lower bound for $\mu_G(\Lambda)$.

Choose $R>0$ small enough for Lemmas~\ref{Lemma: Gap gives subgroup covering} and~\ref{Lemma: Ball between two scales}, with $\overline B_{2R}\subset B_{r_c}$ and with $\overline B_{2R}$ containing no non-trivial subgroup. This choice depends only on $d_G$. It suffices to treat $\delta>0$ smaller than both lemmas' cutoffs and smaller than $1/2$. Let $C$ be the exponent in Lemma~\ref{Lemma: Gap gives subgroup covering} and put $\eta:=\delta^{-C-1}$. Let $m,M,r_a$ be the parameters supplied by Lemma~\ref{Lemma: Ball between two scales}, decreasing its dimensional constant $c$ so that $r_a\leq\min(R,1)$. These parameters depend only on $d_G$ and $\delta$. We choose the measure threshold last.

\smallbreak
We now localise $\Lambda$ so that its square lies in the chosen coordinate and commutator neighbourhoods. Put $\Lambda_0:=\Lambda^2\cap\overline B_R$. We use the covering argument from \cite[\S 2]{machado2020closed}, keeping track of the constants, so we recall a brief outline in our setup. Cover $G$ by $C(d_G)$ sets of diameter at most $R$, and choose a point of $\Lambda$ in each non-empty intersection. This gives $\Lambda\subset Y\Lambda_0$ with $|Y|\leq C(d_G)$. In particular $\Lambda_0$ has positive measure, and $\mu_G(\Lambda_0)\leq K\mu_G(\Lambda)$. To check the approximate subgroup property, cover $\Lambda^4$ by at most $K^3$ right translates of $\Lambda$, and cover $\overline B_{2R}$ by $C(d_G)$ sets of diameter at most $R$. In each non-empty intersection choose $y$. Any other point $x$ of that intersection satisfies $xy^{-1}\in\Lambda^2\cap\overline B_R=\Lambda_0$. Thus $\Lambda_0^2$ is covered by $C(d_G)K^3$ right translates of $\Lambda_0$, and inversion gives the same left covering bound. The set $\Lambda_0$ is compact and symmetric and contains $e$.

We may therefore replace $\Lambda$ by $\Lambda_0$ and $K$ by $C(d_G)K^3$. We relabel these as $\Lambda$ and $K$. From now on $\Lambda\subset\overline B_R$ and $\Lambda^2\subset\overline B_{2R}$. If $\Lambda\subset H_\delta$ for a proper closed connected subgroup $H$, we are done. Assume henceforth that no such $H$ exists. Choose $F$ with $|F|\leq K^{M-1}$ and $\Lambda^M\subset\Lambda F$, and let $r_1>\cdots>r_k>0$ be the scales supplied by Lemma \ref{Lemma: Boundedly many interesting scales}. Thus $k\leq K^{M-1}\leq K^M$. 

We first dispose of two degenerate cases. Suppose $k=0$, i.e. $r_f=0$ for all $f \in F$. As $\Lambda$ is compact the infimum defining $r_f$ is attained, so $e \in \Lambda f$ and $f \in \Lambda^{-1}=\Lambda$. Hence $\Lambda^{M} \subset \Lambda F \subset \Lambda^2$ and, since $M \geq 4$, $\Lambda^2\Lambda^2 \subset \Lambda^{M} \subset \Lambda^2$. Thus $\Lambda^2$ is a subgroup of $G$ contained in $\overline B_{2R}$. Our fixed choice of $R$ gives $\Lambda^2=\{e\}$, contradicting $\mu_G(\Lambda)>0$. So $k \geq 1$, and Corollary \ref{Corollary: Smallest scale is doubled} gives $B_{r_k} \subset \Lambda^2$. If moreover $r_k \geq r_a$, then $B_{r_a} \subset \Lambda^2$, so $\mu_G(\Lambda)\geq K^{-1}\mu_G(B_{r_a})$. Taking the measure threshold smaller than this quantity excludes that case as well. 

Assume therefore $r_k < r_a$. Discarding the scales $r_i \geq r_a$ and relabelling (which leaves $r_k$ unchanged and only decreases $k$) we may and do assume $r_1 < r_a$, and we set $r_0:=r_a$. Note that $r(s)$ is unchanged by this relabelling for every $s \leq r_a$, since it only involves the $r_f$ smaller than $s$. Lemma \ref{Lemma: Boundedly many interesting scales} applied with $s=r_{i-1}$ therefore gives 
\[ \Lambda^{M} \cap B_{r_{i-1}} \subset \Lambda^2\overline{B}_{r_i}, \qquad i=1, \ldots, k,\]
the case $i=1$ included, where $r_0=r_a$ is not itself one of the scales.

\smallbreak
Suppose that, for some $i$, we have
\[
60d_G\eta r_i\leq r_{i-1},\qquad \Lambda^2\cap B_{\eta r_i}\subset B_{r_i}.
\]
Then $r_i\leq r_a/(60d_G\eta)<\delta$, since $r_a\leq1$ and $\eta>\delta^{-1}$. Lemma~\ref{Lemma: Gap gives subgroup covering} covers the localised $\Lambda$ by at most $K^7$ translates of $H_\delta$ for some proper closed connected $H$. Restoring the original approximate subgroup and its initial cover gives at most $C(d_G)K^{21}\leq C(d_G)K^{100}$ translates in terms of the original $K$, as required.

Assume that no such interval exists. For each $i$ with $r_{i-1}>60d_G\eta r_i$, the annulus $B_{\eta r_i}\setminus B_{r_i}$ therefore meets $\Lambda^2$. The inclusion for $\Lambda^M\cap B_{r_{i-1}}$ established above allows us to apply Lemma~\ref{Lemma: Ball between two scales}, giving
\[
B_{r_{i-1}-30d_G\eta r_i}\subset\Lambda^2\overline B_{3d_G\eta r_i}.
\]

\smallbreak
We use these inclusions to pass from a cover of $B_{r_i}$ by translates of $\Lambda$ to a cover of $B_{r_{i-1}}$, with a multiplicative cost independent of $i$. Since $B_{r_k}\subset\Lambda^2$, the iteration starts at the smallest exceptional scale and ends at $r_0=r_a$.

Let $i \in \{1,\ldots,k\}$ and suppose that $B_{r_i}$ is covered by $N_i$ translates of $\Lambda$. If $r_{i-1} > 60 d_G \eta r_i$, then $ B_{r_{i-1} -30d_G\eta r_i} \subset \Lambda^2\overline{B}_{3d_G\eta r_i}$ and there is $C_2 >0$ depending on the dimension only such that  $\overline{B}_{3d_G\eta r_{i}}$ is covered by $\eta^{C_2}$ translates of $B_{r_i}$ and $B_{r_{i-1}}$ is covered by $2^{C_2}$ translates of $B_{r_{i-1} -30d_G\eta r_i}$, the latter because $r_{i-1} -30d_G\eta r_i > \frac{r_{i-1}}{2}$. Since balls at $e$ are invariant under conjugation, $\Lambda^2\overline B_s=\overline B_s\Lambda^2$. We therefore find that $ B_{r_{i-1} -30d_G\eta r_i}$ is covered by $\eta^{C_2}N_i$ translates of $\Lambda^3$, hence that $B_{r_{i-1}}$ is covered by $\eta^{2C_2}K^2N_i$ translates of $\Lambda$. If $r_{i-1} \leq 60 d_G \eta r_i$, we also conclude that $B_{r_{i-1}}$ is covered by $\eta^{2C_2}K^2N_i$ translates of $\Lambda$, as $B_{r_{i-1}} \subset B_{60d_G\eta r_i}$ is covered by $\eta^{C_2}$ translates of $B_{r_i}$. Starting from $B_{r_k}\subset\Lambda^2$, we can therefore construct covers with
\[N_k\leq K,\qquad N_{i-1}\leq\eta^{2C_2}K^2N_i\quad(i=k,\ldots,1).\]
Hence $B_{r_0}=B_{r_a}$ is covered by at most
\[N_0\leq K\bigl(\eta^{2C_2}K^2\bigr)^k,\qquad k\leq K^M,\]
translates of $\Lambda$. This is a bound $C_3>1$ depending only on $\delta$, $d_G$ and $K$. Therefore
\[
\mu_G(\Lambda)\geq C_3^{-1}\mu_G(B_{r_a}).
\]
At the localised stage choose the measure threshold $\epsilon$ so that
\[
0<\epsilon<\min\{K^{-1},C_3^{-1}\}\mu_G(B_{r_a}).
\]
Both $r_a$ and $C_3$ depend only on the permitted parameters. The hypothesis $\mu_G(\Lambda)\leq\epsilon$ now contradicts the displayed lower bound, and the same choice excludes the earlier case $r_k\geq r_a$. For the original approximate subgroup, make this choice with the localised constant $C(d_G)K^3$ and divide the resulting threshold by the original $K$, as required by $\mu_G(\Lambda_0)\leq K\mu_G(\Lambda)$.
\end{proof}
Carolino's structure theorem for locally compact approximate groups \cite{carolino2015structure} provides another approach to the inverse problem, but uses model-theoretic compactness arguments. The proof above is finitary and yields quantitative measure thresholds.

Although computable, the dependence of $\epsilon$ on $\delta$, on $d_G$ and on $K$ is far from optimal. The approximation principle (Lemma~\ref{Lemma: Approximation principle}) permits $r_a=c(d_G)m^{-2}$ after decreasing $c(d_G)$ to meet the fixed geometric restrictions. Tracking the constants above gives $C_3 \leq K\left(\eta^{2C_2}K^2\right)^{k}$ with $k \leq K^{M}$, while $M=O_{d_G}(\delta^{-C(d_G)})$ and $\mu_G(B_{r_a})\geq c(d_G)r_a^{d_G}$. Restoring the initial localisation and decreasing the measure threshold by a factor depending only on $d_G$ and $K$, we may therefore take
\[\epsilon=\exp\left(-\exp\left(C(d_G,K)\delta^{-C(d_G)}\right)\right)\]
for all sufficiently small $\delta>0$. Thus the reciprocal threshold is at most doubly exponential in a power of $\delta^{-1}$ when $K$ is fixed.

Our methods are also related to discretised product theorems \cite{zbMATH06466329}. We use the escape estimates from that work in Lemma \ref{Lemma: Generation}. Its main product theorem, however, imposes covering-number bounds at a range of scales, whereas Theorem \ref{Proposition: approximate subgroups of small measure} assumes only small Haar measure. The former therefore does not directly imply the latter.

\subsection{From almost simple to semi-simple}
After passing to the adjoint group, we argue by induction on the number of almost simple factors. If the image of $\Lambda$ in a proper factor quotient has small measure, induction applies. In the remaining case, a fixed power $A:=\Lambda^m$ projects onto both factors of a decomposition $G=H\times H'$, where $H$ is almost simple. The gist is to use stability to show that $A$ lies close to the graph of a group homomorphism $\psi:H'\to H$.

Let $\pi_{H'}:G\to H'$ be the coordinate projection. We can choose a Borel section $s:H'\to A$, written $s(y)=(f(y),y)$, but its values need not respect multiplication. Both the quotients of two choices in a fibre and the multiplicative defects of $s$ belong to $A^3\cap H$. It therefore suffices to show that this intersection lies close to the identity.

The conjugacy estimate (Lemma~\ref{Lemma: Products of a conjugacy class}) supplies the bound needed for this concentration. Since $A$ projects onto $H$, conjugating an element of $A^3\cap H$ by $A$ produces its entire $H$-conjugacy class. An element far from the identity would then force a lower bound for $\mu_G(\Lambda)$, contrary to the small measure assumption. Once $A^3\cap H$ is confined near the identity, the map $f:H'\to H$ is almost multiplicative. Stability gives the homomorphism $\psi$, and the same control of the fibres places $A$ near its graph, a proper connected subgroup of $G$.

\begin{lemma}[Products of a conjugacy class]\label{Lemma: Products of a conjugacy class}
Let $H$ be a compact connected centreless almost simple Lie group of dimension $n$, with a bi-invariant Riemannian metric. There is $c_H>0$ such that, for every $h\in H$,
\[\mu_H((h^H)^n)\geq c_Hd(e,h)^{2n},\]
where $h^H$ is the conjugacy class of $h$.
\end{lemma}

\begin{proof}
The case $h=e$ is immediate, so assume $h\neq e$. We want to show that the product set $(h^H)^n$ contains a ball. Consider the map
\begin{align*}
\Phi:H^n&\longrightarrow H\\
(g_1,\ldots,g_n)&\longmapsto(g_1hg_1^{-1})\cdots(g_nhg_n^{-1}),
\end{align*}
whose image is precisely $(h^H)^n$. We will choose $g_1,\ldots,g_n$ so that the differential of $\Phi$ at this point is surjective. The inverse function theorem then gives a neighbourhood of $\Phi(g_1,\ldots,g_n)$ contained in $(h^H)^n$.

To make this quantitative, we select one direction in which to vary each $g_i$, so that the resulting $n$ derivatives are linearly independent with controlled separation. The estimates below show that these variations produce a ball of radius at least $c\,d(e,h)^2$, where $c>0$ depends only on $H$ and its metric.

\smallbreak
Write $\mathfrak h:=\operatorname{Lie}(H)$, with the norm induced by the metric and the corresponding operator norm. Put $\alpha:=d(e,h)$ and $L:=\Id-\Ad(h)$. The adjoint representation is faithful and has injective differential, so compactness gives $\ell:=\lVert L\rVert \geq c\alpha$. Choose a unit vector $v\in\mathfrak h$ such that $w:=Lv$ has norm $\ell$. Irreducibility of the adjoint representation gives
\[\int_H\Ad(g)w\otimes\Ad(g)w\,d\mu_H(g)=\frac{\ell^2}{n}\Id.\]
Consequently, for every proper subspace $W$ and every fixed $p\in H$, some $g\in H$ satisfies $\operatorname{dist}(\Ad(pg)w,W)\geq\ell/\sqrt n$.

Choose $g_1,\ldots,g_n$ successively. Set $c_i:=g_ihg_i^{-1}$, $p_{i-1}:=c_1\cdots c_{i-1}$ and $X_i:=\Ad(g_i)v$, choosing $g_i$ so that $z_i:=\Ad(p_{i-1}g_i)w$ has distance at least $\ell/\sqrt n$ from the span of its predecessors. The map
\[F(t_1,\ldots,t_n)=\prod_{i=1}^n\exp(t_iX_i)c_i\exp(-t_iX_i)\]
is the restriction of $\Phi$ to the variations $g_i(t_i)=\exp(t_iX_i)g_i$, and takes values in $(h^H)^n$. After right translation to the identity, the columns of $DF(0)$ are $z_1,\ldots,z_n$. Its determinant has absolute value at least $(\ell/\sqrt n)^n$, while its largest singular value is at most $\sqrt n\ell$. Hence its smallest singular value is at least $c_n\ell\geq c\alpha$.

The second derivatives of $t\mapsto\log(F(t)F(0)^{-1})$ are uniformly bounded on a fixed neighbourhood of $0$, by compactness and $\lVert X_i\rVert =1$. Write this map as $Mt+R(t)$, with $R(0)=0$ and $DR(0)=0$. Choose $c'>0$ so small that the closed Euclidean ball $E=\{t\in\mathbb R^n \vert \lVert t\rVert \leq c'\alpha\}$ lies in that neighbourhood and $\lVert M^{-1}DR(t)\rVert \leq1/2$ on $E$. Since $\lVert M^{-1}\rVert \leq C/\alpha$, choose $c''>0$ so that $\lVert y\rVert \leq c''\alpha^2$ implies $\lVert M^{-1}y\rVert \leq c'\alpha/2$. Then $t\mapsto M^{-1}(y-R(t))$ is a contraction of $E$ into itself. The fixed-point theorem shows that the image of $F$ contains a ball of radius $c''\alpha^2$ centred at $F(0)$, whose Haar measure is at least $c_H\alpha^{2n}$.
\end{proof}

\begin{proof}[Proof of Theorem \ref{Proposition: approximate subgroups of small measure}.]
We first reduce to the adjoint group. Let $Z:=Z(G)$, $\overline G:=G/Z$ and $p:G\to\overline G$ be the quotient map. The compact set $p(\Lambda)$ is still $K$-approximate, and
\[0<\mu_G(\Lambda)\leq\mu_{\overline G}(p(\Lambda))=\mu_G(Z\Lambda)\leq|Z|\mu_G(\Lambda).\]
Suppose the conclusion holds in $\overline G$ with a proper closed connected subgroup $L$, and put $M=p^{-1}(L)$. The map $p|_{M^\circ}$ has surjective differential and therefore open image in the connected group $L$. Hence $p(M^\circ)=L$ and $M=M^\circ Z$. For the quotient metric, $p^{-1}(L_\delta)=M_\delta$, so each lifted translate is covered by at most $|Z|$ translates of $(M^\circ)_\delta$. Since $|Z|$ is bounded in terms of $d_G$, this changes only the permitted constants and the measure threshold. The normalisation of the quotient metric is comparable to our convention by dimensional factors.

It therefore suffices to prove the theorem for adjoint groups, which are direct products $G=G_1\times\cdots\times G_n$ of centreless almost simple groups. We proceed by induction on $n$. The case $n=1$ is Proposition \ref{Proposition: approximate subgroups of small measure, simple case}. Assume $n\geq2$.

We will call \emph{factor map} a projection $\pi: G \rightarrow S$ modulo the product of a non-empty proper subset of the almost simple factors of $G$. Such an $S$ is compact semi-simple with at most $n-1$ almost simple factors, and $\ker \pi$ is connected. We record for later use that
\[ \pi^{-1}\left(L_{\delta'}\right)=\left(\pi^{-1}(L)\right)_{\delta'}\]
for every subgroup $L \subset S$ and every $\delta' > 0$, where $S$ carries the quotient distance. Indeed $d(\pi(x),\pi(y))=\inf_{k \in \ker \pi}d(x,yk)$, and $\pi^{-1}(L)$ is invariant under right translation by $\ker \pi$, so that $d(\pi(x),L)=d(x,\pi^{-1}(L))$ for all $x \in G$.

Let $\epsilon > 0$ be small enough that the conclusion of Theorem \ref{Proposition: approximate subgroups of small measure}, for the radius $\delta$ and for $K$-approximate subgroups of measure at most $\epsilon$, holds in each of the finitely many factor quotients of $G$, adjusting the radius by the metric comparison above. Let $\Lambda \subset G$ be a compact $K$-approximate subgroup with $e \in \Lambda$, $\mu_G(\Lambda)>0$ and $\mu_G(\Lambda) \leq \eta$, with $\eta$ to be chosen later.

Let $\pi: G \rightarrow S$ be a factor map. Then $e \in \pi(\Lambda)$, and the quotient integration formula, recorded below as \eqref{Eq: Quotient formula} gives $\mu_S(\pi(\Lambda)) \geq \mu_G(\Lambda) > 0$ since the fibre measures are probability measures. Suppose that $\mu_S(\pi(\Lambda)) \leq \epsilon$ for some factor map $\pi$. Then $\pi(\Lambda)$ is contained in $C(d_G,K)$ translates of $L_{\delta}$ for some proper closed connected subgroup $L \subset S$, so that $\Lambda$ is contained in as many translates of $\pi^{-1}(L_{\delta})=\left(\pi^{-1}(L)\right)_{\delta}$. The subgroup $\pi^{-1}(L)$ is closed, proper (because $L$ is) and connected (because $L$ and $\ker \pi$ are) and we are done. Assume henceforth that $\mu_S(\pi(\Lambda)) > \epsilon$ for every factor map. Since $S$ is connected, Kemperman's inequality \cite{Kemperman64} then ensures that
\[ \pi(\Lambda^m)=S \text{ for every factor map } \pi: G \rightarrow S, \text{ where } m:=\lceil \epsilon^{-1}\rceil+1.\]

Indeed Kemperman's inequality gives full measure, and the compact set $\pi(\Lambda^m)$ is then all of $S$.

\smallbreak
Now write $H:=G_1$ and $H':=G_2\cdots G_n$, so that $\pi_{H'}: G \rightarrow G/H$ and $\pi_H: G \rightarrow G/H'$ are factor maps, and write $d_H:=\dim H$. Given $h \in H$ and $X \subset G$ let $h^X:=\{xhx^{-1} \vert x \in X\}$, a subset of $H$ since $H$ is normal. As $H'$ centralises $H$, the subset $h^X$ depends on $X$ only through $\pi_H(X)$. Since $\pi_H(\Lambda^m)=G/H'$, we have $h^{\Lambda^m}=h^H$. Lemma~\ref{Lemma: Products of a conjugacy class} therefore gives, with $\alpha=d(e,h)$,
\begin{equation}
\mu_H\left((h^{\Lambda^m})^{d_H}\right)\geq c_H\alpha^{2d_H}. \label{Eq: large measure through large quotient}
\end{equation}

Each fibre of $\Lambda^{6d_Hm+1}$ above $\pi_{H'}(\Lambda)$ contains a translate of $\Lambda^{6d_Hm}\cap H$. The quotient formula \eqref{Eq: Quotient formula} therefore gives
\[ \mu_H(\Lambda^{6d_Hm} \cap H) \mu_{G/H}(\pi_{H'}(\Lambda)) \leq \mu_G(\Lambda^{6d_Hm+1}) \leq K^{6d_Hm}\eta,\]
so that $\mu_H(\Lambda^{6d_Hm} \cap H) \leq K^{6d_Hm}\eta\epsilon^{-1}=:\eta'$. But every $\lambda \in \Lambda^{3m} \cap H$ satisfies $\left(\lambda^{\Lambda^m}\right)^{d_H} \subset \Lambda^{5d_Hm} \cap H \subset \Lambda^{6d_Hm} \cap H$, so \eqref{Eq: large measure through large quotient} yields
\[ \Lambda^{3m} \cap H \subset B_r, \qquad r:=2(\eta'/c_H)^{\frac{1}{2d_H}}.\]

\smallbreak
We identify $G/H$ with $H'$, so that $\pi_{H'}$ is the coordinate projection onto $H'$. The restriction $\pi_{H'}|_{\Lambda^m}$ is a continuous surjection between compact metrizable spaces. The Federer--Morse theorem therefore supplies a Borel section $s:H'\to\Lambda^m$ with $\pi_{H'}\circ s=\operatorname{id}_{H'}$, as in \cite[Proof of Lem. 5.16]{machado2020closed}. Write
\[s(y)=(f(y),y),\qquad f:H'\longrightarrow H.\]
The map $f$ is Borel, being the first coordinate of $s$. For $x,y\in H'$,
\[s(xy)s(y)^{-1}s(x)^{-1}
=\bigl(f(xy)f(y)^{-1}f(x)^{-1},e\bigr)
\in\Lambda^{3m}\cap H\subset B_r.\]
Thus $d(f(xy),f(x)f(y))<r$, where $H$ carries the restricted metric. By measurable stability and automatic continuity \cite[Facts 7.23--7.24]{JingTran2021}, there are $c(G)>0$ depending on $G$ only and a continuous homomorphism $\psi:H'\to H$ satisfying
\[d(f(y),\psi(y))\leq2r\qquad(y\in H'),\]
provided $r\leq c(G)$. Let
\[\Gamma_\psi:=\{(\psi(y),y)\vert y\in H'\}\subset H\times H'=G\]
be its graph. Every $\lambda\in\Lambda^m$ satisfies $s(\pi_{H'}(\lambda))^{-1}\lambda\in\Lambda^{2m}\cap H\subset B_r$, so
\[\Lambda\subset\Lambda^m\subset s(H')B_r\subset\Gamma_\psi B_{3r}.\]
The graph $\Gamma_\psi$ is a closed connected subgroup of dimension $\dim H'<d_G$, hence proper. Choose $\eta$ so that $3r\leq\min(\delta,c(G))$.
 \end{proof}

The induction above costs one further exponential per almost simple factor of $G$, since the threshold $\eta$ is chosen from $\epsilon$ through $m = \lceil \epsilon^{-1}\rceil+1$. More precisely, write $E_1(t):=e^t$ and $E_{j+1}(t):=\exp(E_j(t))$. The condition on $r$ is satisfied by taking $\eta$ to be a sufficiently small dimensional multiple of
\[\epsilon\delta^{2d_H}K^{-6d_H(\lceil\epsilon^{-1}\rceil+1)}.\]
Together with the almost simple case, this shows that for fixed $K$ there are $C=C(d_G,K)>0$ and $D=D(d_G)>0$ such that the measure threshold in Theorem \ref{Proposition: approximate subgroups of small measure} may be chosen as
\[\epsilon=\frac{1}{E_{n+1}(C\delta^{-D})}\]
for all sufficiently small $\delta>0$, when $G$ has $n$ almost simple factors. The finite central quotient and the comparisons of metrics only change the constants. We believe that replacing the stability theorem used above \cite{GroveKarcherRuh74,JingTran2021} with a suitable stability result for compact approximate subgroups should improve this dependence. We do not pursue this here.

\section{Expansion of small subsets}\label{Section: Expansion of small subsets}

In this section, we prove Theorem \ref{Theorem: BM in compact groups} and, thus, Theorem \ref{Theorem: Expansion inequality}. The first part of the proof uses Theorem \ref{Proposition: approximate subgroups of small measure} to cover $A$ and $B$ by boundedly many translates of a small neighbourhood of a proper subgroup. The second combines the local Brunn--Minkowski theorem (Theorem \ref{Theorem: Local Brunn--Minkowski}) with double counting to obtain the bound near that subgroup. A final comparison of pieces then gives the inequality for $A$ and $B$.

\subsection{Doubling of balls}
We will use the following small-ball asymptotic, whose leading term is determined by the dimension.
\begin{lemma}\label{Lemma: Growth in balls}
Let $G$ be a compact connected Lie group of positive dimension $d_G$ and let $d$ be a bi-invariant Riemannian distance. There is $c>0$ such that, as $\rho\downarrow0$,
\[\mu_G(B_G(e,\rho))=c\rho^{d_G}\left(1-\frac{\operatorname{Scal}_G}{6(d_G+2)}\rho^2+o(\rho^2)\right).\]
Consequently,
\[\frac{\mu_G(B_G(e,2\rho))}{\mu_G(B_G(e,\rho))} = 2^{d_G}(1 - S\rho^2 + o(\rho^2)),\]
where $S=\operatorname{Scal}_G/(2(d_G+2))\geq0$ and $\operatorname{Scal}_G$ is the scalar curvature of the chosen metric. 
\end{lemma}

This can be seen from the formula for the volume of balls from scalar curvature (e.g. \cite[Fact 2.7]{jing2023measure}) and the fact that bi-invariant metrics on compact Lie groups have non-negative scalar curvature, which vanishes precisely in the flat case \cite[Thm 2.2]{zbMATH03532133}.  

\subsection{Double counting}
Let $\delta, \rho> 0$ and let $H \subset G$ be a proper closed subgroup. Let $B_H(h, \rho)$ (resp. $B_{G/H}(\overline{g}, \delta)$) denote the ball of radius $\rho$ at $h$ (resp. radius $\delta$ at $\overline{g}$) with respect to the restriction of $d$ to $H$ (resp. the projection of $d$ to $G/H$). We call a Borel subset $T_{\delta} \subset G$ a \emph{cross-section} of $B_{G/H}(eH,\delta)$ if the restriction of the projection $\pi: G \rightarrow G/H$ is a bijection from $T_{\delta}$ onto $B_{G/H}(eH, \delta)$. Write $\mathfrak{h}$ for the Lie algebra of $H$ and $\mathfrak{h}^{\perp}$ for its orthogonal complement in $\mathfrak{g}$. We will always take 
\[T_{\delta}:= \exp \left( B_{\mathfrak{h}^{\perp}}(0,\delta)\right),\]
the image through the exponential map of the ball of radius $\delta$ centred at $0$ in $\mathfrak{h}^{\perp}$.

This is a cross-section of $B_{G/H}(eH,\delta)$ contained in $B_G(e,\delta)$ as soon as $\delta$ is smaller than the injectivity radius of $G/H$ at $eH$. Indeed, $\mathfrak{h}^{\perp}$ is $\Ad(H)$-invariant, so that $\pi$ is a Riemannian submersion with horizontal distribution $g \mapsto d_eL_g(\mathfrak{h}^{\perp})$ once $G/H$ is equipped with the projection of $d$. Given $X \in \mathfrak{h}^{\perp}$, the one-parameter subgroup $t \mapsto \exp(tX)$ is a geodesic of $G$ (as $d$ is bi-invariant) whose velocity $d_eL_{\exp(tX)}(X)$ is horizontal at all times, so it projects to a geodesic of $G/H$ of the same speed \cite[Prop. 2.109]{GallotHulinLafontaine}. In other words, the restriction of $\pi \circ \exp$ to $\mathfrak{h}^{\perp}$ is the Riemannian exponential map of $G/H$ at $eH$ read through the linear isometry $d_e\pi_{|\mathfrak{h}^{\perp}} : \mathfrak{h}^{\perp} \rightarrow T_{eH}(G/H)$, and the claim follows. 

Write $\mathcal R(h,\rho,\delta):=T_\delta B_H(h,\rho)$ for these rectangles. For fixed $H$ and $\delta$, and a Borel subset $X\subset H_\delta$, define the \emph{slice}
\[
X_{h,\rho}:=X\cap\mathcal R(h,\rho,\delta),\qquad h\in H.
\]
The notation $X_{h,\rho}$ suppresses the fixed subgroup and transverse radius. We use it whenever these are determined by the surrounding argument.

\begin{lemma}\label{Lemma: Double counting}
Let $G$ be a compact connected Lie group equipped with a bi-invariant Riemannian distance $d$ and its Haar probability measure $\mu_G$, and let $H \subset G$ be a proper closed subgroup. Equip $H$ with its Haar probability measure $\mu_H$ and $G/H$ with the $G$-invariant probability measure $\mu_{G/H}$. Let $\delta, \rho > 0$ and let $T_{\delta} \subset B_G(e,\delta)$ be a cross-section of $B_{G/H}(eH,\delta)$. Then, for every Borel subset $X \subset H_{\delta}$,
\begin{equation}
    \int_H \mu_G(X \cap T_\delta B_{H}(h, \rho)) d\mu_H(h) = \mu_G(X)\mu_H(B_{H}(e, \rho)). \label{Eq: Double counting}
\end{equation}
\end{lemma}

The proof of \eqref{Eq: Double counting} relies on the quotient integration formula for Haar measure \cite[Ch. 1]{PrincipleHarmonic09}
\begin{equation}
    \mu_G(X) = \int_{G/H} \mu_H(g^{-1}X \cap H) d\mu_{G/H}(gH). \label{Eq: Quotient formula}
\end{equation}

\begin{proof}[Proof of Lemma \ref{Lemma: Double counting}.]
For each point $tH\in B_{G/H}(eH,\delta)$, let $t\in T_\delta$ be its unique representative. The inverse of the Borel bijection $\pi|_{T_\delta}$ is Borel, as is the inverse of the Borel bijection $(t,h)\mapsto th$ from $T_\delta\times H$ to $H_\delta$. Thus the rectangles and fibre integrands below are Borel measurable. Outside this quotient ball we may use any representative, since the corresponding fibres of $X$ are empty. 

Fix $h \in H$. Since $B_H(h,\rho) \subset H$, we have $T_{\delta}B_H(h,\rho) \subset \pi^{-1}(B_{G/H}(eH,\delta))$, so that 
\[ tH \cap T_{\delta}B_H(h,\rho) = \emptyset \text{ whenever } tH \notin B_{G/H}(eH,\delta).\]
If, on the other hand, $tH \in B_{G/H}(eH,\delta)$, then for all $t' \in T_{\delta}$ and $b \in B_H(h,\rho)$ we have $t^{-1}t'b \in H$ if and only if $t^{-1}t' \in H$, that is, if and only if $t'=t$. Hence $t^{-1}T_{\delta}B_H(h,\rho) \cap H = B_H(h,\rho)$ and therefore 
\[ t^{-1}\left(X \cap T_{\delta}B_H(h,\rho)\right) \cap H = t^{-1}X \cap B_H(h,\rho).\]
Applying the quotient formula \eqref{Eq: Quotient formula} to $X \cap T_{\delta}B_H(h,\rho)$, the two displays above yield 
\begin{align}
    \int_H \mu_G(X \cap T_\delta & B_{H}(h, \rho)) d\mu_H(h) \\& = \int_H \int_{B_{G/H}(eH,\delta)} \mu_H(t^{-1}X \cap B_{H}(h,\rho)) d\mu_{G/H}(tH) d\mu_H(h) \label{Eq: DCl1}\\
    & = \int_{B_{G/H}(eH,\delta)} \int_H \mu_H(t^{-1}X \cap B_{H}(h,\rho))d\mu_H(h)d\mu_{G/H}(tH) \label{Eq: DCl2}\\
    & = \mu_H(B_{H}(e,\rho)) \int_{B_{G/H}(eH,\delta)} \mu_H(t^{-1}X \cap H)  d\mu_{G/H}(tH) \label{Eq: DCl3}\\
    & = \mu_H(B_{H}(e,\rho)) \int_{G/H} \mu_H(t^{-1}X \cap H)  d\mu_{G/H}(tH) \label{Eq: DCl4}\\
    & = \mu_G(X)\mu_H(B_{H}(e, \rho)). \label{Eq: DCl5}
\end{align}
In these equalities, we use Tonelli's theorem to obtain \eqref{Eq: DCl2} and the fact that the set function 
\[ Y \longmapsto \int_H \mu_H(Y \cap B_{H}(h,\rho))d\mu_H(h)\] 
on Borel subsets of $H$ is a Haar measure of total mass $\mu_H(B_{H}(e,\rho))$ to go from \eqref{Eq: DCl2} to \eqref{Eq: DCl3}. The step from \eqref{Eq: DCl3} to \eqref{Eq: DCl4} is where the assumption $X \subset H_{\delta}$ is used: since $H_{\delta}=\pi^{-1}(B_{G/H}(eH,\delta))$, the subset $t^{-1}X \cap H$ is empty for every $tH \notin B_{G/H}(eH,\delta)$, so enlarging the domain of integration changes nothing. The quotient formula again yields \eqref{Eq: DCl5}. 
\end{proof}

We also record that, for fixed $\rho$, the slice-mass function $h\mapsto\mu_G(X\cap T_\delta B_H(h,\rho))$ is continuous. Indeed, the quotient formula expresses it as convolution on $H$ of a bounded fibre density with the indicator of $B_H(e,\rho)$, whose translates are continuous in $L^1(H)$. It therefore attains its maximum on compact $H$. By inner regularity, we may apply the local Brunn--Minkowski inequality to measurable slices of positive measure as well as to compact subsets.

\subsection{Multiplicative properties of rectangles}\label{Subsection: Multiplicative properties of rectangles}
The rectangles $\mathcal R(h,\rho,\delta)$ from the preceding subsection also satisfy a useful product inclusion. The map
\begin{align*}
    \mathfrak{h}^{\perp} \times \mathfrak{h} &\longrightarrow G \\
    (X,Y) &\longmapsto \exp(X)\exp(Y)
\end{align*}
is a local diffeomorphism whose differential at $(0,0)$ is the identification $(X,Y)\mapsto X+Y$ with $\mathfrak g$. It sends $B_{\mathfrak{h}^{\perp}}(0,\delta) \times B_{\mathfrak{h}}(0,\rho)$ to $\mathcal{R}(e,\rho,\delta)$ diffeomorphically as soon as $\delta$ and $\rho$ are sufficiently small. 

Here $B_{\mathfrak{h}^{\perp}}(0,\delta)$ and $B_{\mathfrak{h}}(0,\rho)$ denote the corresponding norm balls. The rectangle $\mathcal{R}(e,\rho,\delta)$ is invariant under conjugation by every element of $H$. We use this invariance to control products of rectangles:

\begin{lemma}\label{Lemma: Product of rectangles}
Let $h_1,h_2\in H$. There is a dimensional constant $c(d_G)>0$ such that, for sufficiently small $\delta,\rho_1,\rho_2>0$,
\[\mathcal R(h_1,\rho_1,\delta)\mathcal R(h_2,\rho_2,\delta)\subset\mathcal R(h_1h_2,\rho_1+\rho_2+c(d_G)\delta^2,2\delta).\]
The intersection of the product with $H_\delta$ is contained in $\mathcal R(h_1h_2,\rho_1+\rho_2+c(d_G)\delta^2,\delta)$.
\end{lemma}

\begin{proof}
Since $\mathcal R(e,\rho,\delta)$ is invariant under conjugation by $H$, we may factor out $h_1h_2$ and work at $e$. Write two elements as $e^{X_1}k_1$ and $e^{X_2}k_2$, with $X_i\in B_{\mathfrak h^\perp}(0,\delta)$ and $k_i\in B_H(e,\rho_i)$. Then
\[e^{X_1}k_1e^{X_2}k_2=e^{X_1}e^{\Ad(k_1)X_2}k_1k_2.\]
The two exponents on the right lie in $\mathfrak h^\perp$ and have norm less than $\delta$. Write their product $e^{X_1}e^{\Ad(k_1)X_2}$ as $e^Ye^Z$, with $Y\in\mathfrak h^\perp$ and $Z\in\mathfrak h$. The local diffeomorphism described above gives this factorisation. Note also that the differential of the coordinate map at $0$ is the identity. Moreover, the BCH formula, and the vanishing of the $\mathfrak h$-component of $X_1+\Ad(k_1)X_2$, give $\lVert Z\rVert \leq c(d_G)\delta^2$. Also $\lVert Y\rVert <2\delta$, since the product lies in $H_{2\delta}$ and $\lVert Y\rVert $ is its distance to $H$. Thus
\[e^Zk_1k_2\in B_H(e,\rho_1+\rho_2+c(d_G)\delta^2),\]
which proves the first inclusion. If the original product lies in $H_\delta$, uniqueness of the same tubular coordinates gives $\lVert Y\rVert <\delta$, proving the second. The constants in the coordinate estimates are uniform over the subalgebras of the fixed metric Lie algebra, while the allowed radii may depend on $H$.
\end{proof}

\subsection{Lower bound on expansion close to a subgroup}

We start with a doubling estimate, which also serves as a warm-up for the two-set inequality. The local Brunn--Minkowski inequality gives a factor close to $2^{d_G}$, while the ratio of slice volumes in $H$ tends to $2^{-d_H}$. This is where the codimension $d_G-d_H$ appears:
 \begin{proposition}\label{Proposition: Lower bound on expansion close to a subgroup}
 Let $\epsilon > 0$ and $H \subset G$ be a proper closed subgroup of dimension $d_H$. There is $\delta > 0$ such that if $A, B \subset G$ are two compact subsets with  $A,B,AB \subset H_{\delta}$, then
 \[\mu_G(AB) \geq \left(2^{d_G-d_H} - \epsilon \right)\min\{\mu_G(A),\mu_G(B)\}.\]
 Moreover, for all sufficiently small $\delta>0$, with the cutoff allowed to depend on $H$, the same conclusion holds with $\epsilon=C(d_G)\delta^{\frac43}$ for some dimensional constant $C(d_G)>0$.
 \end{proposition} 
 
\begin{proof}
  If either measure is zero, the assertion is immediate. Let $\rho>0$ be chosen later. The mass of the slices of both $A$ and $B$ attain their maxima by the observation following Lemma \ref{Lemma: Double counting}. First consider the case where the largest $B$-slice is at least as large as the largest $A$-slice. Choose $h_0$ such that $\mu_G(B_{h_0,\rho})$ is maximal. Thus $\mu_G(B_{h_0,\rho}) \geq \mu_G(A_{h,\rho})$ for all $h \in H$.  The rectangles at $h$ and $h_0$ lie in balls of radius $\rho+\delta$ about those points. After translating the first slice on the left by $h^{-1}$ and the second on the right by $h_0^{-1}$, which preserves their measures and the measure of their product, Theorem \ref{Theorem: Local Brunn--Minkowski} asserts that \[(2^{d_G}-\frac{\epsilon}{2} )\mu_G(A_{h,\rho}) \leq \mu_G\left(A_{h,\rho}B_{h_0,\rho}\right)\]
  when $\max(\delta, \rho)$ is sufficiently small. For  $A$-slices of measure $0$ we use the trivial lower bound $0$, so no non-emptiness assumption on these slices is needed. So, by Lemma \ref{Lemma: Double counting}, 
  \begin{align}
     (2^{d_G} - \frac{\epsilon}{2}) \mu_G(A) \mu_H(B_{H}(e, \rho)) = (2^{d_G} - \frac{\epsilon}{2}) \int_H \mu_G\left(A_{h,\rho}\right)d\mu_H(h) \label{Eq: LBl1}\\
       \leq \int_H \mu_G\left(A_{h,\rho}B_{h_0,\rho}\right)d\mu_H(h) \label{Eq: LBl2}\\
       \leq \int_H \mu_G((AB)_{hh_0,2\rho+c(d_G)\delta^2})d\mu_H(h) \label{Eq: LBl3}\\
       = \mu_G(AB)\mu_H(B_{H}(e, 2\rho + c(d_G)\delta^2)), \label{Eq: LBl4}
  \end{align}
where we have used \eqref{Eq: Double counting} applied to $A$ to obtain \eqref{Eq: LBl1}, the inclusion $AB \subset H_{\delta}$ to go from \eqref{Eq: LBl2} to \eqref{Eq: LBl3}, and \eqref{Eq: Double counting} applied to $AB$ to conclude \eqref{Eq: LBl4}. So 
\[ \mu_G(AB) \geq (2^{d_G} - \frac{\epsilon}{2}) \frac{\mu_H(B_{H}(e, \rho))}{\mu_H(B_{H}(e, 2\rho + c(d_G)\delta^2))}\mu_G(A).\]
For sufficiently small radii, balls in $H$ centred at $e$ lie in $H^\circ$, and the Haar normalisation cancels in their volume ratios. Thus Lemma \ref{Lemma: Growth in balls} applies when $d_H>0$, while for $d_H=0$ these ratios equal $1$. Choosing $\rho = \delta^{\frac23}$ therefore yields 
\[ \mu_G(AB) \geq (2^{d_G-d_H} - \epsilon) \mu_G(A)\]
for $\delta$ sufficiently small. If the largest $A$-slice is larger instead, fix that slice and integrate the $B$-slices. Products remain in the order $AB$, their centres are $h_0h$, and the same calculation gives the bound with $\mu_G(B)$ on the right. Either bound implies the assertion with the minimum.

For the moreover part, the argument loses on two counts. Theorem \ref{Theorem: Local Brunn--Minkowski} is applied to subsets of $\mathcal{R}(h,\rho,\delta)$, which is contained in a ball of radius $C(d_G)(\rho+\delta)$, and therefore contributes an error $C(d_G)(\rho+\delta)^2$. Separately, for $\delta^2\leq\rho$ and sufficiently small radii, Lemma \ref{Lemma: Growth in balls} applied in $H$ gives
\[\frac{\mu_H(B_{H}(e,\rho))}{\mu_H(B_{H}(e,2\rho+c(d_G)\delta^2))} = 2^{-d_H}\left(1+\frac{c(d_G)\delta^2}{2\rho}\right)^{-d_H}\left(1+O(\rho^2)\right),\]
which differs from $2^{-d_H}$ by $O_{d_G}(\rho^2+\delta^2/\rho)$. The scalar curvature of $H^\circ$ is bounded in terms of $d_G$, since its Lie bracket and metric are restrictions of those of $G$. The implicit constant may therefore be dimensional once the radii are sufficiently small for $H$. Balancing the two errors by choosing $\rho=\delta^{2/3}$ gives an error $O_{d_G}(\delta^{4/3})$.
\end{proof}

Allowing different radii for the two families of slices gives a Brunn--Minkowski-type inequality:

 \begin{proposition}\label{Proposition: BM close to a subgroup}
 Let $\epsilon > 0$ and $H \subset G$ be a proper closed subgroup of dimension $d_H$. There is $\delta > 0$ such that if $A, B \subset G$ are two non-empty compact subsets such that $A,B, AB \subset H_{\delta}$, then
 \[\mu_G(AB)^{\frac{1}{d_G-d_H}} \geq \left(1 - \epsilon \right)\left(\mu_G(A)^{\frac{1}{d_G-d_H}} + \mu_G(B)^{\frac{1}{d_G-d_H}}\right).\]
 \end{proposition} 

\begin{proof}
    The proof follows Proposition \ref{Proposition: Lower bound on expansion close to a subgroup}, with different radii for the two families of slices. Their ratio, denoted by $t$, lets us compensate for the unequal masses of $A$ and $B$ in the directions tangent to $H$. We may assume $0<\epsilon<1$. 
    Put $a:=\mu_G(A)$, $b:=\mu_G(B)$, $d:=d_G$, $h:=d_H$ and $q:=d-h$. If one measure is zero, the result is immediate. If $\min(a,b)\leq\epsilon^q\max(a,b)$, the result follows from $\mu_G(AB)\geq\max(a,b)$ and $(1-\epsilon)(1+\epsilon)\leq1$. Otherwise set
    \[t:=(b/a)^{1/q},\qquad \epsilon<t<\epsilon^{-1}.\]
    Set $\eta:=1-(1-\epsilon)^{q/d}$ and $\gamma:=\eta/2$. Let $\rho>0$ be chosen later. We keep the $A$-slice radius equal to $\rho$ and the $B$-slice radius equal to $t\rho$ throughout the proof.

    Write $V(s):=\mu_H(B_H(e,s))$ and retain the notation $A_{u,\rho}$ and $B_{u,t\rho}$ for the slices, with $u\in H$. Compare the maxima of
    \[\frac{\mu_G(A_{u,\rho})}{aV(\rho)}\quad\hbox{and}\quad\frac{\mu_G(B_{u,t\rho})}{bV(t\rho)}.\]
    These are the slice masses divided by their expected values for a uniformly chosen centre. Suppose first that the second maximum is at least as large as the first, and let $u_0$ attain it. For every positive-measure $A$-slice,
    \[\frac{\mu_G(B_{u_0,t\rho})}{\mu_G(A_{u,\rho})}\geq\frac{bV(t\rho)}{aV(\rho)}.\]
    For sufficiently small $\rho$ and $\delta$, the local theorem, after translating $A_{u,\rho}$ on the left by $u^{-1}$ and $B_{u_0,t\rho}$ on the right by $u_0^{-1}$, gives
    \begin{equation}
    (1-\gamma)\left(\mu_G(A_{u,\rho})^{1/d}+\mu_G(B_{u_0,t\rho})^{1/d}\right)\leq\mu_G(A_{u,\rho}B_{u_0,t\rho})^{1/d}. \label{Eq: Application local BM}
    \end{equation}
    Consequently the product measure is at least
    \[(1-\gamma)^d\mu_G(A_{u,\rho})\left[1+\left(\frac{bV(t\rho)}{aV(\rho)}\right)^{1/d}\right]^d.\]
    For $A$-slices with measure zero, use the lower bound $0$. Integrating this bound as in \eqref{Eq: LBl1}--\eqref{Eq: LBl4}, and using Lemma \ref{Lemma: Product of rectangles} with the two radii, yields
    \[\mu_G(AB)\geq(1-\gamma)^d\frac{aV(\rho)}{V((1+t)\rho+c(d_G)\delta^2)}\left[1+\left(\frac{bV(t\rho)}{aV(\rho)}\right)^{1/d}\right]^d.\]
    If the first normalised maximum is larger, fix an $A$-slice attaining it and integrate the $B$-slices, still multiplying in the order $AB$. The same argument gives the last display with $aV(\rho)$ and $bV(t\rho)$ interchanged. Thus in either case
    \[\mu_G(AB)^{1/d}\geq(1-\gamma)\frac{(aV(\rho))^{1/d}+(bV(t\rho))^{1/d}}{V((1+t)\rho+c(d_G)\delta^2)^{1/d}}.\]

    Now set $\rho=\delta$. Then the small-ball volume ratios in $H$ converge to their Euclidean ratios uniformly for $t\in[\epsilon,\epsilon^{-1}]$. The extra thickening has relative size $O(\delta^2/\rho)=O(\delta)$. Shrinking $\delta$ in terms of $G,H,\epsilon$ therefore gives
    \[\mu_G(AB)^{1/d}\geq(1-\eta)\frac{a^{1/d}+(bt^h)^{1/d}}{(1+t)^{h/d}}.\]
    The same statement holds for $d_H=0$, when all sufficiently small balls in $H$ have the same measure. At our chosen value $t=(b/a)^{1/q}$, the fraction is exactly $(a^{1/q}+b^{1/q})^{q/d}$. Raising to the power $d/q$ gives the claimed loss $(1-\eta)^{d/q}=1-\epsilon$.
\end{proof}

When $d_H>0$, differentiating the last fraction with respect to $t$ shows that $t=(b/a)^{1/(d_G-d_H)}$ maximises it. When $d_H=0$, the fraction is independent of $t$. Thus the optimisation over the relative slice radii yields exactly the codimension exponent.

The local transport argument does not directly extend to a uniform global estimate with exponent $d_G$. For the squared Riemannian distance cost and $A=B$, the optimal map is the identity almost everywhere, so the product map is $F(x)=x^2$. After left translation to the identity, its differential is $\Id+\Ad(x^{-1})$. This is singular whenever $\Ad(x)$ has eigenvalue $-1$, which occurs in compact almost simple groups. Thus the squaring map does not supply the uniform positive Jacobian bound used in the local proof. The double-counting argument above addresses this obstruction by combining local transport estimates along a subgroup.

\subsection{The Brunn--Minkowski inequality in compact groups}

Theorems \ref{Theorem: Expansion inequality} and \ref{Theorem: BM in compact groups} are stated for compact connected Lie groups, whereas the preceding inverse and near-subgroup arguments use the semi-simple convention. The following lemma shows that nothing is lost.

\begin{lemma}\label{Lemma: Reduction to the semi-simple case}
    Let $G$ be a compact connected Lie group of dimension $d_G$ which is not semi-simple, and let $d_M$ denote the maximal dimension of a proper closed subgroup. Then $d_G - d_M = 1$ and, for all non-empty compact subsets $A, B \subset G$,
    \[ \mu_G(AB) \geq \min\left(\mu_G(A)+\mu_G(B), 1\right).\]
    Consequently Theorems \ref{Theorem: Expansion inequality} and \ref{Theorem: BM in compact groups} hold for $G$.
\end{lemma}

\begin{proof}
    Since $G$ is not semi-simple it has a non-trivial connected closed normal abelian subgroup. A normal torus in a connected group is central, since conjugation gives a continuous map to its discrete automorphism group. Thus $Z:=Z(G)^\circ$ is a torus of dimension at least $1$. Now $G = Z[G,G]$ with $[G,G]$ closed and semi-simple and $Z \cap [G,G]$ finite, so $G/[G,G]$ is a torus of dimension $\dim Z \geq 1$. Composing $G \rightarrow G/[G,G]$ with a surjective character of that torus gives a surjective continuous homomorphism $\chi: G \rightarrow \mathbb{S}^1$, whose kernel is a proper closed subgroup of dimension $d_G-1$. As $G$ is connected, every proper closed subgroup has dimension strictly less than $d_G$, so $d_M = d_G-1$.

    The displayed inequality is Kemperman's \cite{Kemperman64}, valid in every compact connected group. Since $d_G-d_M=1$, the conclusion of Theorem \ref{Theorem: BM in compact groups} reads $\mu_G(AB) \geq (1-\epsilon)\left(\mu_G(A)+\mu_G(B)\right)$, which follows at once when $\alpha < \frac12$, as then $\mu_G(A)+\mu_G(B) < 1$. The conclusion of Theorem \ref{Theorem: Expansion inequality} reads $\mu_G(A^2) \geq (2-\epsilon)\mu_G(A)$, which follows from $\mu_G(A^2) \geq 2\mu_G(A)$, itself an instance of the displayed inequality once $\alpha < \frac12$.
\end{proof}

In the proof of Theorem~\ref{Theorem: BM in compact groups}, we will partition $A$ and $B$ using their covers by translates of $H_r$. To sum the product estimates for the resulting pieces, we need these translates to remain disjoint after enlarging the radius. The following lemma achieves this by merging translates whose enlargements intersect. Each merger removes a centre, so the number of original translates controls the total enlargement.

\begin{lemma}\label{Lemma: Disjoint translates}
    Let $G$ be a group and let $(B_r)_{r>0}$ be a family of symmetric subsets of $G$ containing $e$ such that $B_rB_s \subset B_{r+s}$ for all $r,s>0$. Let $A \subset G$, let $r_0 > 0$ and let $n \geq 3$ be an integer. If there is a non-empty finite subset $F \subset G$ such that $A \subset FB_{r_0}$, then there are $F' \subset F$ and a non-negative integer $m \leq 2(|F|-1)$ such that 
    \[ A \subset \bigsqcup_{f \in F'} fB_{n^{m}r_0}\]
    and $f_1 B_{n^{m+1}r_0} \cap f_2 B_{n^{m+1}r_0} = \emptyset$ for all $f_1 \neq f_2 \in F'$. 
\end{lemma}

\begin{proof}
Since $e \in B_s$ for all $s>0$ we have $B_r \subset B_rB_s \subset B_{r+s}$, so the family $(B_r)_{r>0}$ is increasing. We use this repeatedly without further mention. 

We construct inductively subsets $F=F_0 \supset F_1 \supset \cdots$ and integers $0=m_0 < m_1 < \cdots$ such that $A \subset F_iB_{n^{m_i}r_0}$ for all $i$. First, $A \subset FB_{r_0} = F_0B_{n^{m_0}r_0}$. Suppose $F_i$ and $m_i$ have been constructed. If the subsets $(fB_{n^{m_i+1}r_0})_{f \in F_i}$ are pairwise disjoint, we stop. Otherwise there are $f_1 \neq f_2 \in F_i$ and $b_1, b_2 \in B_{n^{m_i+1}r_0}$ with $f_1b_1=f_2b_2$. As $B_{n^{m_i+1}r_0}$ is symmetric, 
\[ f_2 = f_1b_1b_2^{-1} \in f_1B_{n^{m_i+1}r_0}B_{n^{m_i+1}r_0} \subset f_1B_{2n^{m_i+1}r_0},\]
so that 
\[ f_2B_{n^{m_i}r_0} \subset f_1B_{\left(2n^{m_i+1}+n^{m_i}\right)r_0} \subset f_1B_{n^{m_i+2}r_0},\]
where the second inclusion holds because $2n+1 \leq n^2$ for $n \geq 3$. Setting $F_{i+1}:=F_i \setminus \{f_2\}$ and $m_{i+1}:=m_i+2$ we obtain $A \subset F_{i+1}B_{n^{m_{i+1}}r_0}$. 

Each step removes one element of $F$, so the construction stops after at most $|F|-1$ steps, at some $F'=F_i$ and $m=m_i \leq 2(|F|-1)$. By construction $A \subset F'B_{n^mr_0}$ and the subsets $(fB_{n^{m+1}r_0})_{f \in F'}$ are pairwise disjoint. Since $B_{n^mr_0} \subset B_{n^{m+1}r_0}$, so are the subsets $(fB_{n^{m}r_0})_{f \in F'}$, and the union above is therefore disjoint. 
\end{proof}

Applying Lemma \ref{Lemma: Disjoint translates} to $A^{-1}$ and taking inverses gives the corresponding statement for right translates.

\begin{proof}[Proof of Theorem \ref{Theorem: BM in compact groups}.]
By Lemma \ref{Lemma: Reduction to the semi-simple case}, we may assume that $G$ is semi-simple. We may also assume $0<\epsilon<1$. Let $A,B$ be non-empty compact subsets with $\mu_G(A),\mu_G(B)\leq\alpha$, where $\alpha$ will be chosen below. Put $q:=d_G-d_M$, $a:=\mu_G(A)$ and $b:=\mu_G(B)$. If one measure is zero, a translate of the other set proves the result. Replacing $(A,B)$ by $(B^{-1},A^{-1})$ if necessary, which replaces $AB$ by its inverse, assume $a\geq b>0$. If $b\leq\epsilon^qa$, then $(1-\epsilon)(a^{1/q}+b^{1/q})\leq a^{1/q}\leq\mu_G(AB)^{1/q}$. We may therefore assume $\tau=b/a\in[\epsilon^q,1]$.

If $\mu_G(AB)\geq2^qa$, the desired inequality is again immediate. Otherwise
\[\frac{\mu_G(AB)}{\sqrt{ab}}\leq2^q\tau^{-1/2}\leq2^q\epsilon^{-q/2}.\]
Proposition \ref{Proposition: Tao} gives a compact $L$-approximate subgroup $\Lambda$, with $L=L(d_G,\epsilon)$, and finite sets $F_1'',F_2''$ of size at most $L$, such that $A\subset F_1''\Lambda$, $B\subset\Lambda F_2''$ and $0<\mu_G(\Lambda)\leq L\sqrt{ab}\leq L\alpha$. Here $\alpha$ is the measure threshold, still to be chosen.

The covering number will be fixed before the radius. We can then choose the radius uniformly over maximal proper closed connected subgroups, and choose the measure threshold last.

Fix an integer $N=N(d_G,\epsilon)$ bounding the number of translates after combining these covers with Theorem \ref{Proposition: approximate subgroups of small measure}. This number is independent of the requested neighbourhood radius. There are finitely many conjugacy classes of maximal proper closed connected subgroups of $G$ \cite[\S 2 and \S 4.7]{BurnessLiebeckShalev2020}. By bi-invariance, choose $\delta>0$ so that Proposition \ref{Proposition: BM close to a subgroup} holds in $H_{2\delta}$ for every such $H$, with loss $\zeta=1-(1-\epsilon)^{1/d_G}$. Its conclusion implies
\[\mu_G(XY)\geq(1-\epsilon)\left(\mu_G(X)^{1/(d_G-d_H)}+\mu_G(Y)^{1/(d_G-d_H)}\right)^{d_G-d_H}\]
for non-empty compact $X,Y$ with $X,Y,XY\subset H_{2\delta}$. Indeed $(1-\zeta)^{d_G-d_H}=(1-\epsilon)^{(d_G-d_H)/d_G}\geq1-\epsilon$. By inner regularity we also use this for measurable pieces of positive measure.

Put $r:=3^{-4N}\delta$. Choose $\alpha$ small enough that Theorem \ref{Proposition: approximate subgroups of small measure} applies to $\Lambda$ at radius $r$. Enlarge its subgroup to a maximal proper closed connected subgroup $H$, if necessary. Symmetry of $\Lambda$ gives both a left and a right cover by $H_r$. Combining with the covers of $A,B$ yields
\[A\subset F_1'H_r,\qquad B\subset H_rF_2',\qquad |F_1'|,|F_2'|\leq N.\]
Write $d_H:=\dim H$ and $q_H:=d_G-d_H$. We only know $d_H\leq d_M$, and retain $q_H$ until the end.

Apply Lemma \ref{Lemma: Disjoint translates} once, to
\[A\cup B^{-1}\subset(F_1'\cup(F_2')^{-1})H_r,\]
with $n=3$. This is allowed because $H_rH_s\subset H_{r+s}$. The lemma gives a finite set $F$ and $m\leq4N-2$ such that, with $\delta':=3^mr\leq\delta$,
\[
A\cup B^{-1}\subset\bigcup_{f\in F}fH_{\delta'}.
\]
The enlarged neighbourhoods $(fH_{3\delta'})_{f\in F}$ are pairwise disjoint. Taking inverses gives a right cover of $B$ by $(H_{\delta'}f^{-1})_{f\in F}$ with the same disjointness property at radius $3\delta'$.

Let $F_1\subset F$ and $F_2\subset F^{-1}$ index the positive-measure pieces
\[A_f:=A\cap fH_{\delta'},\qquad B_g:=B\cap H_{\delta'}g.\]
Discarding the null pieces leaves $\sum_{f\in F_1}\mu_G(A_f)=a$ and $\sum_{g\in F_2}\mu_G(B_g)=b$. Both families of enlarged neighbourhoods, $(fH_{3\delta'})_{f\in F_1}$ and $(H_{3\delta'}g)_{g\in F_2}$, are pairwise disjoint. These covers and this radius remain fixed in both cases below.

We use the same comparison of normalised masses as in Proposition \ref{Proposition: BM close to a subgroup}, now dividing the measure of each piece by the measure of its whole set. A piece attaining the maximum over both families can then be used against every piece of the other set.

Choose $g_0\in F_2$ with $\mu_G(B_{g_0})$ maximal. Suppose first that
\[\frac{\mu_G(B_{g_0})}{b}\geq\max_{f\in F_1}\frac{\mu_G(A_f)}{a}.\]
Then
\begin{align}
\mu_G(AB)&\geq\mu_G(AB_{g_0}) \label{Eq: BMc1}\\
&\geq\sum_{f\in F_1}\mu_G(A_fB_{g_0}) \label{Eq: BMc2}\\
&\geq(1-\epsilon)\sum_{f\in F_1}\left[\mu_G(A_f)^{1/q_H}+\mu_G(B_{g_0})^{1/q_H}\right]^{q_H} \label{Eq: BMc3}\\
&=(1-\epsilon)\sum_{f\in F_1}\mu_G(A_f)\left[1+\left(\frac{\mu_G(B_{g_0})}{\mu_G(A_f)}\right)^{1/q_H}\right]^{q_H} \label{Eq: BMc4}\\
&\geq(1-\epsilon)\sum_{f\in F_1}\mu_G(A_f)\left[1+(b/a)^{1/q_H}\right]^{q_H} \label{Eq: BMc5}\\
&=(1-\epsilon)(a^{1/q_H}+b^{1/q_H})^{q_H}. \label{Eq: BMc6}
\end{align}
For \eqref{Eq: BMc2}, the products lie in the pairwise disjoint sets $fH_{2\delta'}g_0$. For \eqref{Eq: BMc3}, apply the local subgroup inequality to $f^{-1}A_f$ and $B_{g_0}g_0^{-1}$, whose product lies in $H_{2\delta'}$. The comparison of normalised maxima gives \eqref{Eq: BMc5}, and the sum of the piece measures gives \eqref{Eq: BMc6}.

In the other case, choose $f_0\in F_1$ maximising $\mu_G(A_f)/a$ and sum over $g\in F_2$. The products $A_{f_0}B_g$ lie in the pairwise disjoint sets $f_0H_{2\delta'}g$. The choice of $f_0$ gives
\[\frac{\mu_G(A_{f_0})}{\mu_G(B_g)}\geq\frac{a}{b}\qquad(g\in F_2).\]
Thus the same local inequality gives
\[\mu_G(AB)\geq(1-\epsilon)\sum_{g\in F_2}\mu_G(B_g)\left[1+(a/b)^{1/q_H}\right]^{q_H}=(1-\epsilon)(a^{1/q_H}+b^{1/q_H})^{q_H}.\]
In either case,
\begin{equation}
\mu_G(AB)\geq(1-\epsilon)(a^{1/q_H}+b^{1/q_H})^{q_H}. \label{Eq: BMc7}
\end{equation}
For $a,b>0$, the function $t\mapsto(a^{1/t}+b^{1/t})^t$ is non-decreasing, by monotonicity of the $\ell^p$-quasi-norms. Since $q_H\geq q$, we conclude that
\[\mu_G(AB)\geq(1-\epsilon)(a^{1/q}+b^{1/q})^q.\]
Taking the $q$-th root and using $(1-\epsilon)^{1/q}\geq1-\epsilon$ proves the theorem.
\end{proof}

\begin{proof}[Proof of Theorem \ref{Theorem: Expansion inequality}.]
    Let $\epsilon > 0$ and choose $\epsilon'\in(0,1)$ with $(1-\epsilon')^{d_G-d_M}2^{d_G-d_M} \geq 2^{d_G-d_M}-\epsilon$. Let $\alpha > 0$ be given by Theorem \ref{Theorem: BM in compact groups} applied with $\epsilon'$.
    The case $A=\emptyset$ is immediate. For non-empty $A$, taking $A=B$ there yields $\mu_G(A^2)^{\frac{1}{d_G-d_M}} \geq (1-\epsilon')2\mu_G(A)^{\frac{1}{d_G-d_M}}$ for every compact $A$ with $\mu_G(A) \leq \alpha$, which is the claim.
\end{proof}

\begin{proof}[Quantitative bound following Theorem \ref{Theorem: Expansion inequality}.]
Assume that $G$ is semi-simple with $n$ almost simple factors and put $q=d_G-d_M$. We may assume $\mu_G(A^2)<2^q\mu_G(A)$, since otherwise the bound is immediate. Apply Proposition \ref{Proposition: Tao} with $B=A$. Its parameter $K$ is now bounded by $2^q$, so the resulting approximate subgroup constant and both covering numbers depend only on $d_G$. The number $N$ in the proof of Theorem \ref{Theorem: BM in compact groups} can therefore be fixed independently of the desired error.

\smallbreak
Let $\delta>0$ be sufficiently small. Use the same simultaneous disjoint covers at radius $\delta'\leq\delta$, starting with the inverse theorem at radius $3^{-4N}\delta$. Enlarge the subgroup to a maximal proper closed connected subgroup $H$. Proposition \ref{Proposition: Lower bound on expansion close to a subgroup}, applied in $H_{2\delta}$, gives a loss at most $C(d_G)\delta^{4/3}$, uniformly over these subgroups because they have finitely many conjugacy classes. Choose a piece of largest measure among the left and right families of pieces of $A$, and multiply it against every piece in the other family, keeping the order of the factors. In each product the minimum of the two measures is the measure of the varying piece. The products are disjoint by the same argument as in \eqref{Eq: BMc2}, so summing gives
\[\mu_G(A^2)\geq\left(2^{d_G-d_H}-C(d_G)\delta^{4/3}\right)\mu_G(A)
\geq\left(2^q-C(d_G)\delta^{4/3}\right)\mu_G(A).\]

\smallbreak
The quantitative threshold at the end of \S \ref{Section: Approximate subgroups of small measure}, with the now fixed approximate subgroup constant, shows that this argument applies whenever
\[0<\mu_G(A)\leq\frac{1}{E_{n+1}(C_0\delta^{-D})}\]
for suitable $C_0,D>0$ depending only on $d_G$. The fixed factor in the radius and the bound $\mu_G(\Lambda)\leq C(d_G)\mu_G(A)$ are absorbed by increasing $C_0$. Put $L:=\log^{\circ(n+1)}(1/\mu_G(A))$ and choose $\delta=(C_0/L)^{1/D}$. For sufficiently small $\mu_G(A)$ this radius satisfies all the preceding restrictions, and the error is at most $C(d_G)L^{-4/(3D)}$. This is the claimed bound with $c=4/(3D)$.
\end{proof}

\section{Closing remarks}\label{Section: Closing remarks}

\subsection{The quest for minimisers}
Our results corroborate the following semi-simple version of the minimiser conjecture raised for almost simple groups in \cite[Conj. 3.13]{jing2023measure}: 

\begin{conjecture}
    Let $G$ be a compact connected semi-simple Lie group of positive dimension, equipped with the metric fixed in \S \ref{Subsection: Notation}. For all sufficiently small $\alpha>0$ there are a proper closed subgroup $H$ of dimension $d_M$ and $\delta>0$, with $2\delta$ below the injectivity radius of the induced metric on $G/H$, such that $\mu_G(H_\delta)=\alpha$ and
    \[\mu_G(A^2)\geq\mu_G((H_\delta)^2)\]
    for every measurable subset $A\subset G$ with $\mu_G(A)=\alpha$.
\end{conjecture}

The subgroup $H$ is chosen as part of the minimisation and may depend on $\alpha$. If true, this conjecture would pin down the exact rate at which the doubling constant approaches $2^{d_G-d_M}$, and in particular the correcting term discussed after Theorem \ref{Theorem: Expansion inequality}.

Inverse results can describe sets with small doubling without identifying exact minimisers, as in Freiman-type structure theorems \cite{Sanders13Survey}. The existence and structure of minimisers therefore remain separate questions, even when an inverse theorem is available:

\begin{problem}
      Let $G$ be a connected Lie group with a fixed left Haar measure $\mu_G$, and let $0<\alpha<\infty$ with $\alpha\leq\mu_G(G)$. Is the infimum of $\mu_G(A^2)/\mu_G(A)$ over measurable subsets $A\subset G$ with $\mu_G(A)=\alpha$ attained? When it is, what is the structure of the minimisers?
\end{problem}

Concrete groups $G$ of particular interest are $\SO_3(\mathbb{R})$ and $\SL_2(\mathbb{R})$. Note however that such minimisers never exist in the Heisenberg group $\mathcal{H}_3(\mathbb{R})$ \cite{zbMATH03318704}. 

A related quantitative question concerns Proposition \ref{Proposition: Lower bound on expansion close to a subgroup}. For a fixed proper closed subgroup $H$ and all sufficiently small $\delta>0$, can one replace the loss $C(d_G)\delta^{\frac43}$ by $C(d_G)\delta^2$, under the same hypotheses on $A,B,AB$? At scales $\rho\geq\delta$, the present proof loses $O(\rho^2+\delta^2/\rho)$. The term $\delta^2/\rho$ comes from the tangential enlargement of the output rectangle to radius $2\rho+c(d_G)\delta^2$, and the choice $\rho=\delta^{\frac23}$ balances the two errors. A quadratic bound would require going beyond this balance, perhaps by estimating the effect of the tangential enlargement more precisely.

\subsection{Growth of small subsets}\label{Subsection: Babai's conjecture for compact Lie groups}

Lower bounds on the growth of products are one route to diameter bounds for Cayley graphs, as in the questions surrounding Babai's conjecture \cite{zbMATH03528264}. In the compact setting, the following proposition bounds the number of products needed for a small set to reach a fixed positive measure. The size of this threshold is relevant to connecting our estimates with the large-set regime studied in \cite{ellis2024product}.

\begin{proposition}\label{Proposition: Growth of small subsets}
Let $G$ be a compact connected Lie group of positive dimension and let $0<\delta<1$. There is $\alpha = \alpha(d_G,\delta) > 0$ such that for every compact subset $A \subset G$ with $0<\mu_G(A) \leq \alpha$ we have $\mu_G(A^m) \geq \alpha$ for some $m \leq 2\mu_G(A)^{-\frac{1}{(1-\delta)(d_G-d_M)}}$.
\end{proposition}

\begin{proof}
Apply Theorem \ref{Theorem: Expansion inequality} with $\epsilon := 2^{d_G-d_M}-2^{(1-\delta)(d_G-d_M)}$, which is positive since $\delta > 0$, and let $\alpha \in (0,1]$ be the resulting threshold. It may be chosen to depend only on $d_G$ and $\delta$: Lemma \ref{Lemma: Reduction to the semi-simple case} gives a uniform threshold when $G$ has an abelian factor, and there are only finitely many semi-simple groups of dimension $d_G$ up to isomorphism. That theorem then gives
\[\mu_G(B^2) \geq 2^{(1-\delta)(d_G-d_M)}\mu_G(B)\]
for every compact subset $B \subset G$ with $\mu_G(B) \leq \alpha$.

Write $\nu:=\mu_G(A)$ and let $k \geq 0$ be least such that $\mu_G(A^{2^k}) \geq \alpha$. Such a $k$ is well defined since the measures $\mu_G(A^{2^j})$ increase by a factor at least $2^{(1-\delta)(d_G-d_M)}>1$ as long as they stay below $\alpha$, while remaining at most $1$. If $k=0$ the conclusion holds with $m=1$, so assume $k \geq 1$. For every $j<k$ we have $\mu_G(A^{2^j}) < \alpha$ by minimality of $k$, so the displayed inequality applies to $B = A^{2^j}$ and, by induction, $\mu_G(A^{2^j}) \geq 2^{j(1-\delta)(d_G-d_M)}\nu$ for all $j \leq k$. Applying this with $j = k-1$ and using $\mu_G(A^{2^{k-1}}) < \alpha \leq 1$ yields $2^{k-1} \leq \nu^{-\frac{1}{(1-\delta)(d_G-d_M)}}$, hence $m=2^k$ satisfies the required bound.
\end{proof}

Improving the measure threshold in Theorem~\ref{Proposition: approximate subgroups of small measure} would sharpen both the quantitative doubling bound and the growth estimate above. The quadratic error near a subgroup asked for in the preceding subsection is a separate obstruction in the present proof.

\section*{Acknowledgements}
I am deeply grateful to Yifan Jing and Chieu Minh Tran for many fascinating discussions and their encouragement. I am indebted to Alessio Figalli, Federico Glaudo and Lauro Silini for sharing their insights on optimal transport and the Brunn--Minkowski inequality. It is my pleasure to thank Peter van Hintum and Nicolas de Saxcé for many insightful comments on a first version of this work.

ChatGPT was used to assist with revisions to the exposition and bibliography.

\appendix
    \section{Subgroups of maximal dimension}\label{Section: Subgroup of maximal dimension}
We list the Lie algebras $\mathfrak g$ of compact almost simple Lie groups and a Lie algebra type $\mathfrak h_M$ of a proper closed subgroup of maximal dimension. The table records Lie algebra types rather than conjugacy classes of subgroups. We use the compact real forms, with $\mathfrak a_r=\mathfrak{su}(r+1)$, $\mathfrak b_r=\mathfrak{so}(2r+1)$, $\mathfrak c_r=\mathfrak{sp}(r)$ and $\mathfrak d_r=\mathfrak{so}(2r)$. The entries follow \cite[Table 2.1]{biller1999actions}, which is based on \cite[\S 4]{10.1215/ijm/1256055005}.
\begin{table}[htbp]
\caption{Lie algebra types and codimensions of proper subgroups of maximal dimension in compact almost simple Lie groups.}
\label{Table: Maximal subgroup codimensions}
\centering
    \[
    \begin{array}{|l|l|l|l|} 
    \hline
\mathfrak{g}   &   \dim \mathfrak{g}  &  \dim(\mathfrak{g}/\mathfrak{h}_M) & \mathfrak{h}_M   \\
    \hline
 \mathfrak{a}_r (r\geq1,\ r\neq 3) & r(r+2)      & 2r  & \mathbb{R} \times \mathfrak{a}_{r-1} \\
  \mathfrak{b}_r (r\geq 3) & r(2r+1)      & 2r  & \mathfrak{d}_{r} \\
   \mathfrak{c}_r (r\geq 2) & r(2r+1)      & 4(r-1)  & \mathfrak{a}_1 \times \mathfrak{c}_{r-1} \\
   \mathfrak{d}_r (r\geq 3) & r(2r-1)      & 2r-1  &  \mathfrak{b}_{r-1} \\
   \mathfrak{e}_6  & 78      & 26  &  \mathfrak{f}_{4} \\
   \mathfrak{e}_7  & 133      & 54  &  \mathbb{R} \times \mathfrak{e}_{6} \\
   \mathfrak{e}_8  & 248      & 112  &  \mathfrak{a}_1 \times \mathfrak{e}_{7} \\
   \mathfrak{f}_4  & 52      & 16  &   \mathfrak{b}_{4} \\
   \mathfrak{g}_2  & 14      & 6  &   \mathfrak{a}_{2} \\
    \hline
    \end{array}
\]
\end{table}
Here $\mathfrak a_0=\{0\}$, and the low-rank coincidences are read through $\mathfrak a_1\simeq\mathfrak b_1\simeq\mathfrak c_1$, $\mathfrak b_2\simeq\mathfrak c_2$ and $\mathfrak a_3\simeq\mathfrak d_3$.

\end{document}